\documentclass[oneside,english]{amsart}
\usepackage[T1]{fontenc}
\usepackage[latin9]{inputenc}
\usepackage{textcomp}
\usepackage{amsthm}
\usepackage{esint}

\makeatletter
\numberwithin{equation}{section}
\numberwithin{figure}{section}
\theoremstyle{plain}
\newtheorem{thm}{\protect\theoremname}[section]
  \theoremstyle{definition}
  \newtheorem{example}[thm]{\protect\examplename}
  \theoremstyle{plain}
  \newtheorem{prop}[thm]{\protect\propositionname}
  \theoremstyle{plain}
  \newtheorem{lem}[thm]{\protect\lemmaname}
  \theoremstyle{remark}
  \newtheorem{rem}[thm]{\protect\remarkname}
  \theoremstyle{plain}
  \newtheorem{cor}[thm]{\protect\corollaryname}
  \theoremstyle{definition}
  \newtheorem{defn}[thm]{\protect\definitionname}

\def\makebbb#1{
    \expandafter\gdef\csname#1\endcsname{
        \ensuremath{\Bbb{#1}}}
}\makebbb{R}\makebbb{N}\makebbb{Z}\makebbb{C}\makebbb{H}\makebbb{E}\makebbb{H}\makebbb{P}\makebbb{B}\makebbb{Q}\makebbb{E}

\usepackage{babel}

\usepackage{babel}

\makeatother

\usepackage{babel}

\makeatother

\usepackage{babel}
  \providecommand{\corollaryname}{Corollary}
  \providecommand{\definitionname}{Definition}
  \providecommand{\examplename}{Example}
  \providecommand{\lemmaname}{Lemma}
  \providecommand{\propositionname}{Proposition}
  \providecommand{\remarkname}{Remark}
\providecommand{\theoremname}{Theorem}

\begin{document}

\title{complex optimal transport and the pluripotential theory of Kähler-Ricci
solitons }

\author{Robert J. Berman, David Witt Nyström}
\begin{abstract}
Let $(X,L)$ be a (semi-) polarized complex projective variety and
$T$ a real torus acting holomorphically on $X$ with moment polytope
$P.$ Given a probability density $g$ on $P$ we introduce a new
type of Monge-Ampère measure $MA_{g}(\phi)$ on $X,$ defined for
singular $T-$invariant metrics $\phi$ on the line bundle $L,$ generalizing
the ordinary Monge-Ampère of global pluripotential theory, which corresponds
to the case when $T$ is trivial (or $g=1).$ In the opposite extreme
case when $T$ has maximal rank, i.e. $(X,L,T)$ is a toric variety,
the solution $\phi$ of the corresponding Monge-Ampère equation $MA_{g}(\phi)=\mu$
corresponds to the convex Kantorovich potential for the optimal transport
map in the Monge-Kantorovich transport problem betweeen $\mu$ and
$g$ (with a quadratic cost function). Accordingly, our general setting
can be seen as a complex version of optimal transport theory. Our
main complex geometric applications concern the pluripotential study
of singular (shrinking) Kähler-Ricci solitons. In particular, we establish
the uniqueness of such solitons, modulo automorphisms, and explore
their relation to a notion of modified K-stability inspired by the
work of Tian-Zhu. The quantization of this setup, in the sense of
Donaldson, is also studied.

\tableofcontents{}
\end{abstract}
\maketitle

\section{Introduction}

\subsection{Background and motivation}

Ever since the seminal work of Yau \cite{y} and Aubin \cite{au}
on Kähler-Einstein metrics on complex manifolds, i.e. Kähler metrics
with constant Ricci curvature, complex Monge-Ampère equations have
played a central role in complex geometry. Recall that a Riemannian
metric on a complex manifold $X$ with complex structure $J$ is said
to be Kähler when it can be written as $\omega(\cdot,J\cdot)$ for
a closed two-form $\omega$ on $X,$ which equivalently means that
$\omega$ can be locally written as 
\[
\omega=\omega_{\phi}:=\frac{i}{2\pi}\partial\bar{\partial}\phi,
\]
for a local function $\phi,$ which is strictly plurisubharmonic,
i.e. the complex Hessian $\frac{\partial^{2}\phi}{\partial z_{i}\partial\bar{z_{j}}}$
is positive. In the case when $\omega$ has integral periods, i.e.
$[\omega]\in H^{2}(X,\Z),$ the local functions $\phi$ patch to define
a Hermitian metric on a positive/ample line bundle $L$ with curvature
form $\omega_{\phi},$ representing the first Chern class $c_{1}(L)$
in $H^{2}(X,\Z).$ In particular, Yau's solution of the Calabi conjecture
concerning the existence of a Ricci flat Kähler metric $\omega$ on
a Calabi-Yau manifold $X$ amounts (in the case when $[\omega]\in H^{2}(X,\Z)$)
to the solvability of a complex Monge-Ampère equation, which in local
notation may be formulated as 
\begin{equation}
\det(\frac{\partial^{2}\phi}{\partial z_{i}\partial\bar{z_{j}}})=f\label{eq:local ma eq intro}
\end{equation}
for $f$ a given positive smooth density on the $n-$dimensional complex
manifold $X$ (similarly, the equation for non-Ricci flat Kähler-Einstein
metrics is obtained by replacing $f$ with $e^{\pm\phi}).$ In global
terms the equation \ref{eq:local ma eq intro} thus prescribes the
volume form, i.e. the top exterior power $\omega_{\phi}^{n}$ of the
Kähler metric $\omega_{\phi}.$

Subsequently Bedford and Taylor \cite{b-t1} developed the local pluripotential
theory which, in particular, furnishes a notion of weak solution to
the highly non-linear complex Monge-Ampère equation is a \emph{(singular)
Kähler-Ricci soliton} on $X$ if the metric\ref{eq:local ma eq intro}
by making sense of the the wedge products ($i\partial\bar{\partial}\phi)^{n},$
as long as $i\partial\bar{\partial}\phi$ is positive in the sense
of currents and $\phi$ is locally bounded. The local work of Bedford-Taylor
and its extension to compact Kähler manifolds by Kolodziej \cite{ko1}
and Guedj-Zeriahi \cite{g-z} was generalized to a very general global
complex geometric frame work in \cite{begz}. In particular, using
the non-pluripolar product of positive currents defined in \cite{begz}
this allows one to define the complex Monge-Ampère measure 
\[
MA(\phi):=(\frac{i}{2\pi}\partial\bar{\partial}\phi)^{n}
\]
for \emph{any} (possibly singular) positively curved metric $\phi$
on a big line bundle $L\rightarrow X.$ As shown in \cite{begz} this
leads to very general existence and uniqueness results for global
complex Monge-Ampère equations of the form \ref{eq:local ma eq intro},
by using pluripotential capacity techniques to reduce the situation
to the original setting of Aubin and Yau. A direct \emph{variational}
approach to complex Monge-Ampère equations and Kähler-Einstein metrics
was recently introduced in \cite{begz,bbgz,be2}, which can be seen
as a non-linear version of the classical Dirichlet energy variational
principle for the Laplace equation on a Riemann surface.

There is also a\emph{ real} version of this story. Indeed, as is well-known,
in the case when the plurisubharmonic function $\phi(z)$ is independent
of the imaginary part of $z$ - i.e. $\phi(z)=\varphi(x)$ for a convex
function $\varphi$ on $\R^{n}$ - the complex Monge-Ampère measure
may be identified with the \emph{real} Monge-Ampère measure $MA_{\R}(\varphi)$
of the convex function $\varphi.$ The latter measure was geometrically
defined in the classical works by Alexandrov and Pogorelov by using
the multivalued map defined by the subgradient $\partial\varphi$
of $\varphi:$ the mass $MA(\varphi)(E)$ of a Borel set $E$ in $\R^{n}$
is the Lesbegue volume of the image of $E$ under $\partial\varphi.$
This real situation appears naturally in the global complex geometric
framework when $(X,L)$ is\emph{ toric,} i.e. when there is an action
of the full complex torus $T_{c}^{n}$ on $(X,L)$ - then any metric
$\phi$ on $L$ which is invariant under the action of the corresponding
real torus $T^{n}$ is naturally identified with a convex function
$\varphi$ on $\R^{n}.$ However, an important flexibility which arises
in the setting of the real Monge-Ampère equation is that, for any
given (say, continuous) non-negative function $g(p)$ on the space
$\R^{n},$ invariantly viewed as the dual real vector space, the product
\begin{equation}
MA_{g}(\varphi)_{\R}:=MA(\varphi)_{\R}g(\nabla\varphi)\label{eq:g monge amp intro motiv}
\end{equation}
is well-defined as a measure, as long as $\varphi$ is convex (indeed,
compared with the previous definition one simply replaces the Lesbegue
measure $dp$ on the dual $\R^{n}$ by $g(p)dp)$. In particular,
this situation appears naturally in the theory of \emph{optimal (mass)
transport, }originating in the classical works of Monge and Kantorovich.
The point is that the corresponding Monge-Ampère equation
\[
MA_{g}(\varphi)_{\R}=fdx
\]
 is equivalent to the mass density $f(x)$ being transported optimally
by the $L_{loc}^{\infty}-$ map defined by gradient $\nabla\varphi$
to the density $g(p)dp:$ 

\[
(\nabla\varphi)_{*}(fdx)=g(p)dp,
\]
As shown by Brenier \cite{br} the existence of such a convex function
$\varphi$ follows from variational considerations (involving Kantorovich
duality), using that $\varphi$ minimizes the Kantorovich cost functional.
In fact, as observed in \cite{berm7} the complex and real variational
approaches in \cite{bbgz} and \cite{br}, respectively, are essentially
equivalent in the toric setting. In particular, the\emph{ cost} in
the real setting corresponds, in the complex setting, to\emph{ energy.}

One of the main general aims of the present paper is to consider the
general situation where a torus $T$ - not necessarily of maximal
rank - acts on $(X,L)$ and define a generalized version of the $g-$Monge-Ampère
measure in formula \ref{eq:g monge amp intro motiv}, for any $T-$invariant,
possibly singular, metric $\phi$ on $L$ with positive curvature
current. This situation can thus be seen as a hybrid of the complex
and the real settings for the Monge-Ampère equation and leads to a
complex generalization of the theory of optimal transport (this point
will be expanded on elsewhere). Our starting point is the basic observation
that for a smooth metric $\phi$ on $L$ there is natural generalization
of the gradient map, namely the\emph{ moment map} $m_{\phi}$ determined
by the metric $\phi$ in terms of symplectic geometry, which defines
a map from $X$ into the dual of the Lie algebra of $T$ (depending
on the local first derivatives of $\phi$ along the torus orbits).
Inspired by some ideas originating in some recent work on the connection
between filtrations and pluripotential theory one one hand and test
configurations and geodesic rays in Kähler geometry on the other \cite{w,r-n1,hi},
we are led to a canonical definition of the generalized $g-$Monge-Ampère
measure, which has good continuity properties. 

Our main motivation for developing this general framework is to provide
a pluripotential theoretic notion of a weak solution to the \emph{Kähler-Ricci
soliton equation} on a, possibly singular, Fano variety $X,$ where
$L$ is the anti-canonical line bundle $-K_{X},$ i.e. the top exterior
power of the holomorphic tangent bundle of $X.$ In the ordinary smooth
case a Kähler metric $\omega$ is a Kähler-Ricci soliton precisely
when it can be written as the curvature form $\omega_{\phi}$ of a
smooth metric $\phi$ on $-K_{X},$ locally satisfying the complex
Monge-Ampère equation 
\begin{equation}
\det(\frac{\partial^{2}\phi}{\partial z_{i}\partial\bar{z_{j}}})e^{f^{\phi}}=e^{-\phi},\label{eq:local krs eq intro}
\end{equation}
 where $f^{\phi}$ is the Hamiltonian function corresponding to the
imaginary part of a holomorphic vector field $V$ on $X,$ generating
an action of a real torus $T$ on $(X,-K_{X})$ (i.e. the orbits of
the torus $T$ coincide with the closure of the orbits of flow of
the imaginary part of $V).$ Since the Hamiltonian may be expressed
as $f^{\phi}=\left\langle m_{\phi},\xi\right\rangle ,$ where $\xi$
is the element in the Lie algebra of $T$ corresponding to $V,$ the
right hand side in the equation \ref{eq:local krs eq intro} is the
density of $MA_{g_{V}}(\phi)$ for $g_{V}(\cdot)=\exp\left\langle \cdot,\xi\right\rangle .$
This observation will allow us to define the notion of a weak solution
$\phi$ of the Kähler-Ricci soliton equation \ref{eq:local krs eq intro},
a notion which turns out to be very useful for both uniqueness and
existence problems, as well as convergence problem. We also develop
a general ``quantized'' (finite dimensional) version of the $g-$Monge-Ampère
setting, which generalizes Donaldson's setting of balanced metrics
introduced in \cite{do3} and, in particular, leads to a new finite
dimensional analog of a Kähler-Ricci soliton.

\subsection{Pluripotential theory of moment maps and $g-$Monge-Ampère equations}

Let $L$ be a holomorphic line bundle over a compact $n-$dimensional
complex manifold $X$ and assume that $(X,L)$ comes with a holomorphic
action of a real torus $T$ of rank $m$ (which equivalently means
that the $T-$action on $X$ is Hamiltonian). To fix ideas first assume
that $L$ is ample. Then any $T-$invariant smooth positively curved
metric $\phi$ on $L$ induces, via the symplectic form $\omega_{\phi}$
defined by the curvature of $\phi,$ a\emph{ moment map} 
\[
m_{\phi}:\,\,\, X\rightarrow\R^{m},\,\,\,\,\,\, P:=m_{\phi}(X)
\]
 for the $T-$action, where we have identified the Lie algebra of
$T$ (and its dual) with $\R^{m}$ in the standard way. As is well-known
the image $P$ is compact and independent of $\phi$ - more precisely,
$P$ is a convex polytope and thus usually referred to as the \emph{moment
polytope} - and coincides with the support of the corresponding (normalized)
\emph{Duistermaat-Heckman measure} 
\[
\nu:=(m_{\phi})_{*}MA(\phi),\,\,\,\,\,\,\, MA(\phi):=\frac{1}{c_{1}(L)^{n}}\omega_{\phi}^{n}
\]
defining a probability measure on $\R^{m},$ which is absolutely continuous
with respect to Lesbegue measure and independent of $\phi$ \cite{d-h}.

More generally, this setup applies as long as $L$ is \emph{semi-positive
and big,} i.e. $L$ admits some smooth metric $\phi$ with non-negative
curvature form $\omega_{\phi}$ of positive total volume, which will
be assumed henceforth (in particular, by passing to a resolution $X$
may be allowed to be singular). Given a continuous non-negative function
$g$ on $\R^{n},$ or rather on $P$ (usually normalized so that $g\nu$
is a probability measure) we will write 

\begin{equation}
MA_{g}(\phi):=MA(\phi)g(m_{\phi}),\label{eq:gma for smooth intro}
\end{equation}
which defines, for any smooth and $T-$invariant non-negatively curved
metric $\phi,$ a measure on $X$ which will be referred to as the\emph{
$g-$Monge-Ampère measure} (or the $g-$modified Monge-Ampère measure).
As explained above one of the main points of the present paper is
to extend the definition of $MA_{g}(\phi)$ to the space of\emph{
all} (possibly singular) $T-$invariant metrics on $L$ with positive
curvature current and show that it has the same good continuity properties
as in the standard case when $g=1.$
\begin{thm}
\label{thm:g-ma intro}Let $L\rightarrow X$ be a line bundle with
an action of a real torus $T,$ as above, and $g$ a continuous function
on the corresponding moment polytope $P.$ Then there exists a unique
extension of the smooth $g-$Monge-Ampère measure $MA_{g}(\phi)$
defined by formula \ref{eq:gma for smooth intro} to the space of
all $T-$invariant (possibly singular) metrics $\phi$ on $L$ with
positive curvature with the following properties:
\begin{itemize}
\item If $\phi_{j}$ is a sequence of metrics decreasing to a locally bounded
metric $\phi,$ then the corresponding measures $MA_{g}(\phi_{j})$
converge weakly to $MA_{g}(\phi)$ on $X.$ 
\item The measure $MA_{g}(\phi)$ does not charge pluripolar subsets of
$X$
\item The measure $MA_{g}(\phi)$ is local with respect to the $T-$plurifine
topology on $X$
\end{itemize}
Moreover, the following properties also hold:
\begin{itemize}
\item \textup{$\int_{X}MA_{g}(\phi)\leq\int_{P}gd\nu$ }with equality if
and only if the metric $\phi$ has full Monge-Ampère mass (i.e. $\int_{X}MA(\phi)=1).$
\item The convergence statement in the first point above more generally
holds for any decreasing sequence of singular metrics converging to
a metric $\phi$ of full Monge-Ampère mass.
\end{itemize}
\end{thm}
In particular, the previous theorem implies that for any bounded decreasing
sequence $\phi_{j}$ of smooth positively curved metrics on $L$ the
corresponding measures $MA(\phi_{j})g(m_{\phi_{j}})$ have a unique
weak limit on $X,$ which seems hard to prove directly even if the
limiting metric is smooth.

With the previous theorem in hand it is rather straight forward to
adapt the variational approach in \cite{bbgz} to prove the following
result, generalizing the case when $T$ is trivial considered in \cite{bbgz,begz,ko1}: 
\begin{thm}
\label{thm:solv of gma eq intro}Let $\mu$ be a probability measure
on $X$ which is $T-$invariant and assume that $gd\nu$ is a probability
measure on the moment polytope $P.$ Then $\mu$ does not charge pluripolar
sets iff there exists a metric $\phi$ on $L$ with positive curvature
current such that 
\[
MA_{g}(\phi)=\mu
\]
Moreover, the following is equivalent if $g$ is bounded from below
by a positive constant: 
\begin{itemize}
\item The measure $\mu$ has finite (pluricomplex) energy
\item The solution $\phi$ has finite (pluricomplex) energy
\end{itemize}
and in the finite energy case any solution $\phi$ is unique modulo
constants. In particular, if $\mu$ has a density $f$ in $L^{p}(X),$
for some $p>1,$ then the solution $\phi$ is continuous.
\end{thm}
(the notion of pluricomplex energy is recalled in sections \ref{sub:Energy-type-functionals and proj},
\ref{sub:Monge-Amp=0000E8re-equations-and}). In the other extreme
case, i.e. when $T$ has maximal rank $n,$ so that $(X,L)$ is a
polarized toric variety, the previous result, concerning the finite
energy measures $\mu,$ is essentially equivalent to the existence
and uniqueness result of Brenier for optimal transport maps \cite{br}.
The point is that in this situation finite energy corresponds to finite
cost (see \cite{ber-ber,berm7}). Moreover, the case of a general
probability measure $\mu$ can, in the toric situation, be seen as
a variant of a result of McCann \cite{mc1} (who also proves uniqueness).

\subsection{Applications to Kähler-Ricci solitons }

Recall that Hamilton's Ricci flow emanating from an initial Kähler
metric $\omega_{0}$ on a complex manifold $X$ preserves the Kähler
property and can thus (after normalization) be written as the\emph{
Kähler-Ricci flow} 
\[
\frac{d\omega_{t}}{dt}:=-\mbox{Ric \ensuremath{\omega_{t}+\omega_{t}}}
\]
which exists for any positive time $t$ \cite{ca} (here $\mbox{Ric \ensuremath{\omega}}$
denotes, as usual, the Ricci curvature form of a Kähler metric $\omega).$
The fixed points of the (normalized) Kähler-Ricci flow are \emph{Kähler-Einstein
metrics} with positive Ricci curvature and more generally, the fixed
points of the induced flow on the space of all Kähler metrics modulo
automorphisms correspond to (shrinking) \emph{Kähler-Ricci solitons}
on $X,$ i.e. a Kähler metric $\omega$ on $X$ such that there exists
a complex holomorphic vector field $V$ on $X$ with the property
that 
\begin{equation}
\mbox{Ric \ensuremath{\omega=\omega}}+L_{V}\omega,\label{eq:k-r soliton eq for metric omega intro}
\end{equation}
 where $L_{V}$ denotes the Lie derivative of $\omega$ with respect
to $V.$ In particular, the metric $\omega$ is invariant under the
flow generated by the imaginary part $\mbox{Im\ensuremath{V}}$ of
$V,$ which equivalently means that $\mbox{Im\ensuremath{V}}$ generates
a Hamiltonian action of a torus $T$ on $X$ and $L_{V}\omega=dd^{c}f,$
where $f$ is a Hamiltonian function for the flow of $\mbox{Im\ensuremath{V}}.$
As a consequence any manifold $X$ carrying a Kähler-Ricci soliton
is \emph{Fano,} i.e. the anti-canonical line bundle $-K_{X}(:=\Lambda^{n}TX)$
is positive/ample (since $\mbox{Ric \ensuremath{\omega}}$ represents
the first Chern class $c_{1}(-K_{X})).$ According to the Hamilton-Tian
conjecture \cite{ti1,t-zhang} the Kähler-Ricci flow $\omega_{t}$
on a Fano manifold $X$ the Riemannian metric defined by $\omega_{t}$
always (sub)converges in the Gromov-Hausdorff topology to a singular
generalization of a Kähler-Ricci soliton $\omega_{\infty}$ defined
on a singular normal Fano variety $X_{\infty},$ which is a complex
deformation of the original Fano manifold $X.$ 

Motivated by the Hamilton-Tian conjecture and in particular its relation
to K-stability (see below) we will in this paper initiate the pluripotential
study of Kähler-Ricci solitons. First note that there is a natural
differential geometric definition of a Kähler-Ricci soliton $\omega$
on a singular (normal) Fano variety $X$ (compare \cite{bbegz} for
the Kähler-Einstein case): the Kähler metric $\omega$ is defined
on the regular part $X_{reg}$ of $X,$ where it solves the equation
\ref{eq:k-r soliton eq for metric omega intro} for some holomorphic
vector field $V$ on $X_{reg}$ and moreover the volume of the metric
$\omega$ on $X_{reg}$ is maximal in the sense that it coincides
with the global algebraic top intersection number $c_{1}(-K_{X})^{n},$
i.e. the degree of the Fano variety $X$ (abusing terminology slightly
we will also refer to the corresponding pair $(\omega,V)$ as a Kähler-Ricci
soliton). Our first result reveals, in particular, that this definition
coincides with various a priori stronger definitions previously proposed
in the literature \cite{p-s-s,z2,t-zhang} (that impose some information
on the singularities of $X$ and on $\omega$ along the singular locus
of $X$): 
\begin{thm}
\label{thm:krs regularity intro}Let $X$ be a normal Fano variety
admitting a Kähler-Ricci soliton $\omega$ such that the imaginary
part of the corresponding holomorphic vector field $V$ generates
a torus $T.$ Then $X$ has log terminal singularities and the Kähler
metric $\omega,$ originally defined on the regular locus $X_{reg},$
extends to a unique positive current $\bar{\omega}$ on $X$ in $c_{1}(-K_{X})$
with continuous potentials. More precisely, $\bar{\omega}$ is the
curvature current of a continuous $T-$invariant metric $\phi$ on
the line bundle $L:=-K_{X}$ satisfying the global $g_{V}-$Monge-Ampère
equation on $X$ corresponding to the equation \ref{eq:local krs eq intro}.
Conversely, any (singular) $T-$invariant metric $\phi$ on $-K_{X}$
with positive curvature current and full Monge-Ampère mass which is
a weak solution to the equation \ref{eq:local krs eq intro} is continuous
and has a curvature current which is smooth on $X_{reg},$ satisfying
the equation \ref{eq:k-r soliton eq for metric omega intro} there.
In the case that $X$ is a priori assumed to have log terminal singularities
the existence of the torus $T$ is automatic.
\end{thm}
It should be stressed that the regularity statement concerning the
solution $\phi$ in the previous theorem is new even when $X$ is
smooth. Our next result extends the Tian-Zhu uniqueness theorem for
Kähler-Ricci solitons \cite{t-z1,t-z1b} (which in turn generalizes
the Bando-Mabuchi uniqueness theorem for Kähler-Einstein metrics)
to the singular setting. 
\begin{thm}
\label{thm:A-K=0000E4hler-Ricci-soliton is unique intro}A Kähler-Ricci
soliton $(\omega,V)$ on a Fano variety is unique modulo the action
of the group $\mbox{Aut}(X)_{0}$ of all holomorphic automorphism
of $X$ homotopic to the identity. Moreover, a Kähler-Ricci soliton
$\omega$ defined with respect to a\emph{ fixed} holomorphic vector
field $V,$ is unique modulo the subgroup $\mbox{Aut}(X,V)_{0}$ of
$\mbox{Aut}(X)_{0}$ consisting of all automorphism of $X$ commuting
with the flow of $V.$ 
\end{thm}
The proof follows closely the proof of the smooth case in \cite{bern1}
extended to the singular Kähler-Einstein setting in \cite{bbegz},
which proceeds by connecting two given Kähler-Ricci solitons $\phi_{0}$
and $\phi_{1}$ by a weak geodesic curve $\phi_{t}$. The main issue
is that, in the case of a singular Fano variety, the intermediate
metrics $\phi$ are a priori not $C^{1}-$differentiable (even on
the regular locus of $X)$ and this is where the present pluripotential
theory of $g-$Monge-Ampère equations is needed. 

In view of the previous theorem it is natural to fix a holomorphic
vector field $V$ on the Fano variety $X,$ as above, and view the
pair $(X,V)$ as the given complex geometric data, denoting by $\mbox{Aut}(X,V)$
the corresponding automorphism group. If $(X,V)$ admits a Kähler-Ricci
soliton in the sense that $X$ admits a Kähler-Ricci soliton with
corresponding vector field $V,$ then it follows from the previous
theorem, just as in the Kähler-Einstein case considered in \cite{c-d-s}
III, that the group $\mbox{Aut}(X,V)_{0}$ is \emph{reductive }(see
Corollary \ref{cor:reduct}). 

As shown by Tian-Zhu \cite{t-z1b} another necessary condition for
the existence of a Kähler-Ricci soliton on $(X,V)$, in the smooth
case, is the vanishing of the \emph{modified Futaki invariant }introduced
in \cite{t-z1b}, which can be viewed as a functional $\mbox{Fut}_{V}$
on the Lie algebra of $\mbox{Aut}(X,V)_{0}$ and algebraically expressed
as 
\[
\mbox{Fut}_{V}(W)=-\lim_{k\rightarrow\infty}\frac{1}{kN_{k}}\sum_{l=1}^{N_{k}}\exp(v_{l}^{(k)}/k)w_{l}^{(k)},
\]
where $(v_{l}^{(k)},w_{l}^{(k)})$ are the joint eigenvalues for the
commuting action of the real parts of the holomorphic vector fields
$V$ and $W$ on the $N_{k}-$dimensional space $H^{0}(X,-kK_{X})$
of all holomorphic section with values in $-kK_{X}$ (see Proposition
\ref{prop:alg formula for fut}). More generally, in the case $V=0$
there is a notion of algebro-geometric stability of a Fano manifold
$X$ referred to as \emph{K-stability} (or sometimes \emph{K-polystability})
introduced by Tian in \cite{ti1} saying that $X$ is K-polystable
if for any $\C^{*}-$equivariant deformation $\mathcal{X}$ of $X$
the Futaki invariant $\mbox{Fut}(X_{0})$ of the central fiber $X_{0}$
(assumed to have log terminal singularities) satisfies $\mbox{Fut}(X_{0})\geq0$
with equality if and only if $X_{0}$ is biholomorphic to $X.$ Similarly,
in the general case of a pair $(X,V)$ we will say that \emph{$(X,V)$
is K-polystable} if for any $\C^{*}-$equivariant deformation $(\mathcal{X},\mathcal{V})$
of $(X,V)$ the modified Futaki invariant $\mbox{Fut}_{V_{0}}(X_{0})$
of the central fiber $X_{0}$ satisfies $F_{V_{0}}(X_{0})\geq0$ with
equality if and only if $(X_{0},V_{0})$ is isomorphic to $(X,V).$
\begin{thm}
\label{thm:krs implies k-stab intro}Assume that $(X,V)$ admits a
Kähler-Ricci soliton. Then $(X,V)$ is K-polystable.
\end{thm}
This results appears to be new even in the case when $X$ is smooth,
where it generalizes the result of Tian-Zhu in \cite{t-z1b} concerning
product deformations. The proof of the previous theorem builds on
\cite{be4} where the case $V=0$ was considered in the setting of
singular Fano varieties. In the case $V=0$ there is also a generalization
of Tian's notion of K-stability due to Donaldson \cite{d0}, which
involves more general polarized deformations of $X$ (called test
configurations) and which in the end turns out to be equivalent to
Tian's notion. Presumably there is a similar generalization of Donaldson's
notion of K-stability in the presence of a non-trivial vector field
$V,$ as in the setting of extremal Kähler metrics considered in \cite{s-s},
but we will not go further into this here. 

According to the fundamental Yau-Tian-Donaldson conjecture recently
settled in \cite{c-d-s} and \cite{ti} a Fano manifold $X$ is K-polystable
if and only if $X$ admits a Kähler-Einstein metric (the existence
problem in the singular case is still open). It seems natural to conjecture
that this correspondence can be extended to the case of pairs $(X,V),$
with $V$ non-trivial, using the notion of K-stability appearing in
the previous theorem (a different version of this conjecture involving
a notion of geodesic K-stability was recently formulated by He \cite{he}).
In this direction we will show the \emph{analytic} analog of K-polystability
 does imply the existence of a Kähler-Ricci soliton $\omega,$ which
moreover can be realized as the large time limit of the Kähler-Ricci
flow:
\begin{thm}
Let $X$ be a Fano variety and $V$ a holomorphic vector field on
$X$ generating an action on $X$ of a torus $T.$ If $(X,V)$ is
analytically K-polystable, in the sense that the modified Mabuchi
K-energy is proper modulo $\mbox{Aut}(X,V)_{0},$ then $(X,V)$ admits
a Kähler-Ricci soliton. Moreover, the Kähler-Ricci flow $\omega_{t}$
then converges in the weak topology of currents, modulo the action
of the group $\mbox{Aut}(X,V)_{0},$ to the Kähler-Ricci soliton $\omega.$
\end{thm}
In the case when $X$ is smooth the previous existence result was
shown in \cite{c-t-z}, using a variant of Aubin's continuity method
and the convergence result was shown in \cite{t-z2}, using Perelman's
deep estimates \cite{pe}.

It seems reasonable to expect that the recent powerful techniques
developed for the existence problem of Kähler-Einstein metrics (see
\cite{c-d-s,ti,sz} ) - which combine Gromov-Hausdorff convergence
theory with $L^{2}-$estimates for $\bar{\partial}$ - can be extended
to prove the missing algebraic counterpart of the previous Theorem
by a deformation argument; either by using Tian-Zhu's modification
of Aubin's continuity path or the Kähler-Ricci flow (as in the three
dimensional situation considered in \cite{t-zhang}). For some recent
results in this direction see \cite{p-s-s,t-zhang,zw,z2}. As emphasized
in \cite{do stab bir,c-d-s}, in the ordinary case $V=0,$ the reductivity
of the automorphism group of the limiting singular object is an important
step in producing a suitable test configuration from the deformation
in question. Accordingly, we expect the reductivity established in
Corollary \ref{cor:reduct} (resulting from the uniqueness result
in Theorem \ref{thm:A-K=0000E4hler-Ricci-soliton is unique intro}
above) to play a similar role in the case when $V$ is non-trivial
(by replacing the Hilbert scheme used in \cite{c-d-s} with its $T_{c}-$invariant
counterpart).

\subsection{The quantized setting and balanced metrics}

In section \ref{sec:The-quantized-setting} we consider the quantization,
in the sense of Donaldson \cite{do3}, of the setup above. In other
words, given a line bundle $L\rightarrow X,$ the infinite dimensional
space $\mathcal{H}$ of all metrics $\phi$ on $L$ with (semi-)positive
curvature is replaced by the sequence $\mathcal{H}_{k}$ of finite
dimensional symmetric spaces of Hermitian metrics on the complex vector
spaces$H^{0}(X,kL)$ of all holomorphic sections with values in the
$k$ th tensor power of $L,$ written as $kL$ in additive notation.
For example, when $L=-K_{X}$ a \emph{quantized Kähler-Ricci soliton,}
$H_{k}\in\mathcal{H}_{k}$ may in this framework be defined as a fixed
point, modulo $\mbox{Aut}(X,V)_{0},$ of Donaldson's (anti-)canonical
iteration on the symmetric space $\mathcal{H}_{k}.$ As conjectured
in \cite{do3} and confirmed in \cite{be3} the latter iteration can
be viewed as the quantization of the Kähler-Ricci flow. We recall
that the bona fide fixed points in $\mathcal{H}_{k}$ are called (anti-)
canonically \emph{balanced metrics }and as conjectured by Donaldson
\cite{do3}, and shown in \cite{bbgz}, in the case when $X$ admits
a Kähler-Einstein $\omega_{KE}$ and $\mbox{Aut}(X)_{0}$ is trivial
(i.e. $\omega_{KE}$ is unique), such balanced metrics exist for $k$
sufficiently large and the corresponding Bergman metrics $\omega_{k}$
on $X$ converge, as $k\rightarrow\infty,$ weakly to the Kähler-Einstein
metric on $X.$ Here we will introduce a notion of $g-$balanced metrics
which in particular allows us to prove the following generalization
of the result in \cite{bbgz}:
\begin{thm}
\label{thm:strong polyk implies existence and conv of quantized intro}Let
$(X,V)$ be a Fano manifold equipped with a holomorphic vector field
$V.$ If $X$ is strongly analytically K-polystable and all the higher
order modified Futaki invariants of $(X,V)$ vanish, then there exist
quantized Kähler-Ricci solitons $H_{k}$ at any sufficiently large
level $k,$ which are unique modulo the action of $\mbox{Aut}(X,V)_{0}$
and as $k\rightarrow\infty$ the corresponding Bergman metrics $\omega_{k}$
on $X$ converge weakly, modulo automorphisms, to a Kähler-Ricci soliton
$\omega$ on $(X,V).$
\end{thm}
The strong analytic K-polystability referred to above simply means
that the modified Mabuchi functional is coercive modulo $\mbox{Aut}(X,V)_{0},$
which in the case when $\mbox{Aut}(X)_{0}$ is trivial is well-known
to be equivalent to the existence of a unique Kähler-Einstein metric
on $X$ (as well as the properness of the Mabuchi functional). It
should be stressed that the previous theorem is new even in the case
when $V=0,$ if $\mbox{Aut}(X)_{0}$ is non-trivial. The condition
on the vanishing of the higher order Futaki invariants is then equivalent
to the vanishing of Futaki's higher order invariants $\mathcal{F}_{Td^{(m)}}$
\cite{fu} for all integers $m=1,..n,$ which is a necessary condition
for the existence of balanced metrics $H_{k}$ for $k$ large (and
the condition is not implied by the analytic K-polystability assumption
even in the case when $X$ is toric, by the example in \cite{f-o-s}).
Finally, it should be stressed that even in the simplest case $\mbox{Aut}(X)_{0}=0$
the extension of the previous theorem to singular Fano varieties seems
challenging. Indeed, the corresponding existence statement may then
even turn out to be false as indicated by an example in \cite{od},
which reveals that the K-stability of a singular polarized variety
$(X,L)$ does not imply the existence of balanced metrics in the (a
priori) different sense of \cite{d00}. This latter notion of a balanced
metric (or critical metric in the terminology of Zhang) is equivalent
to the asymptotic Chow polystability of $(X,L),$ in the sense of
Geometric Invariant Theory. As shown by Mabuchi \cite{m1} and Futaki
\cite{fu} an analog of the existence part of Theorem \ref{thm:strong polyk implies existence and conv of quantized intro}
holds in the latter setting of balanced metrics: more precisely, the
existence of a constant scalar curvature metric in $c_{1}(L)$ together
with the vanishing of the higher order (ordinary) Futaki invariants
of $(X,L)$ implies the asymptotic Chow polystability of $(X,L),$
but the convergence problem for the corresponding balanced metrics
seems to be open (except in the case case when $\mbox{Aut}(X,L)_{0}=0$
originally settled by Donaldson \cite{d00}).

\subsubsection*{Organization of the paper}

The technical core of the paper is section \ref{sec:The-pluripotential-theory}
where the pluripotential theory of $g-$Monge-Ampère measure $MA_{g}$
is developed. In particular, the continuity properties of $MA_{g}$
(stated in Theorem \ref{thm:g-ma intro}) as well as the continuity
and concavity/convexity properties of the corresponding energy functional
$\mathcal{E}_{g}$ are established. With these results in place the
variational approach to complex Monge-Ampère equations developed in
\cite{bbgz} is not hard to extend to the present setting and the
section is concluded with a proof of Theorem  \ref{thm:solv of gma eq intro}.
Then in section \ref{sec:K=0000E4hler-Ricci-solitons} the pluripotential
theory developed in the previous section is applied to Kähler-Ricci
solitons (by adapting the proofs in \cite{bbegz} concerning singular
Kähler-Einstein metrics) and finally extended to the quantized setting
in section \ref{sec:The-quantized-setting}.

\section{\label{sec:The-pluripotential-theory}The pluripotential theory of
moment maps and $g-$Monge-Ampère equations}

\subsection{Pluripotential preliminaries}

Let $L\rightarrow X$ be a line bundle over a compact complex manifold
$X$ and assume that $L$ is semi-positive and big, i.e. $L$ admits
a smooth metric $\phi_{0}$ with non-negative curvature form $\omega_{0}$
(representing the first Chern class $c_{1}(L)\in H^{2}(X,\Z)$ of
$L)$ and 
\[
c_{1}(L)^{n}:=\int_{X}\omega_{0}^{n}>0
\]
 We will use additive notation for a metric $\phi$ on $L,$ i.e.
given an open set $U\subset X$ and a holomorphic trivializing section
$s_{|U}$ of $L_{|U}$ the metric $\phi$ is represented by the local
function $\phi_{|U}:=\log\left|s\right|_{\phi}^{2},$ where $\left|s\right|_{\phi}$
denotes the length of $s$ wrt the metric $\phi.$ Then the (normalized)
curvature form of the metric $\phi$ may be locally represented as
\[
\omega_{\phi}:=dd^{c}\phi_{|U},\,\,\,\,\,\,\,\, dd^{c}:=\frac{i}{2\pi}\partial\bar{\partial}
\]
 (which, abusing notation slightly will occasionally be written as
$\omega_{\phi}=dd^{c}\phi).$ We will denote by $PSH(X,L)$ the space
of all (possibly singular) metrics $\phi$ on $L$ with positive curvature
current, i.e. $\phi$ is locally upper semi-continuous (usc) and integrable
and $\omega_{\phi}\geq0$ holds in the sense of currents (equivalently,
$\phi$ is locally plurisubharmonic, or psh for short). By the seminal
work of Beford-Taylor the Monge-Ampère (normalized) measure 
\[
MA(\phi):=\omega_{\phi}^{n}/c_{1}(L)^{n}
\]
 is well-defined for any metric $\phi$ in $PSH(X,L)$ which is locally
bounded and defines a probability measure on $X.$ More generally,
the (non-pluripolar) Monge-Ampère measure $MA(\phi)$ can be defined
for any metric $\phi$ in $PSH(X,L)$ by replacing the Bedford-Taylor
wedge products used in the previous formula with the non-pluripolar
product of positive currents introduced in \cite{begz}. This extension
is uniquely determined by the properties that $(i)$ $MA(\phi)$ does
not charge pluripolar subset of $X$ (i.e. sets locally contained
in the $-\infty-$locus of a psh function) and $(ii)$ $MA(\phi)$
is local with respect to the plurifine topology (i.e. the coursest
topology making all psh functions continuous). Concretely, this means
that, 
\begin{equation}
MA(\phi):=\lim_{k\rightarrow\infty}1_{\{\phi^{(k)}>\phi_{0}-k\}}MA(\phi^{(k)}),\,\,\,\,\,\phi^{(k)}:=\max\{\phi,\phi_{0}-k\}\label{eq:def of nonpluripol ma measure-1}
\end{equation}
(by the locality of $MA$ the sequence above is increasing and hence
the limit is indeed well-defined).
\begin{example}
If the metric $\phi$ in $PSH(X,L)$ is locally bounded on a Zariski
open subset $\Omega$ of $X,$ then $MA(\phi)=1_{\Omega}\omega_{\phi}^{n}/c_{1}(L)^{n}.$ 
\end{example}
The Monge-Ampère measure defined by formula \ref{eq:def of nonpluripol ma measure-1}
satisfies $\int_{X}MA(\phi)\leq1$ and accordingly, following \cite{begz},
a metric $\phi$ in $PSH(X,L)$ is said to be of\emph{ maximal Monge-Ampère
mass} if $\int_{X}MA(\phi)=1.$ The subspace of all such metrics is
denoted by $\mathcal{E}(X,L)$ and it has the crucial property that
the Monge-Ampère operator is continuous wrt decreasing $\phi_{j}$
converging to an element $\phi$ in $\mathcal{E}(X,L).$ Note however
that a metric as in the previous example, with non-empty singularity
locus $X-\Omega,$ never has full Monge-Ampère mass. Still we have
the following convergence result, which we will have great use for: 
\begin{prop}
\label{prop:cont of ma with sing}Assume that there exists a Zariski
open subset $\Omega$ such that the metrics $\phi_{j}$ and $\phi$
in $PSH(X,L)$ are locally bounded on $X-\Omega$ and have equivalent
singularities (i.e. $\phi_{j}-\phi$ is bounded, for any fixed $j).$
Then $\int MA(\phi_{j})=\int MA(\phi)$ and if $\phi_{j}$ decreases
to $\phi,$ then $MA(\phi_{j})$ converges weakly to $MA(\phi).$ \end{prop}
\begin{proof}
This is well-known and more generally holds as long as $\phi_{j}$
and $\phi$ have small unbounded locus (in the sense of \cite{begz}).
For completeness we recall the simple argument. First, the equality
of the total masses is a consequence of the fact that $\int MA(\psi)\leq\int MA(\phi)$
if $\psi\leq\phi+C$ and $\phi$ and $\psi$ have small unbounded
locus (see Theorem 1.16 in \cite{begz}.) In fact, in our case this
follows from a simple max construction interpolating between suitable
perturbations of $\phi$ and $\psi.$ Finally, since $MA(\phi_{j})\rightarrow MA(\phi)$
in the local weak topology on $\Omega$ (by Bedford-Taylor theory)
it then follows from the equalities of the total integrals and basic
integration theory that $1_{\Omega}MA(\phi_{j})\rightarrow1_{\Omega}MA(\phi),$
weakly on $X,$ as desired.
\end{proof}
We recall the following generalization (to semi-positive and big line
bundles) of Demailly's classical regularization result concerning
ample line bundles:
\begin{thm}
\label{thm:(regularization)--Any}(regularization) \cite{e-g-z} Any
metric in $PSH(X,L)$ can be written as a decreasing limit of \emph{smooth}
metrics $\phi_{j}$ in $PSH(X,L)$ (i.e. $\phi_{j}\in\mathcal{H}(X,L)).$
\end{thm}
As explained in the introduction of the paper one of our main aims
is to define a modified version of the Monge-Ampère measure, in the
presence of a torus action.

\subsection{\label{sub:The-torus-setting}The torus setting and the $g-$Monge-Ampère
measure}

Let $T$ be an $m-$dimensional real torus acting holomorphically
on $(X,L).$ The complexified torus $T_{c}$ then also acts holomorphically
on $(X,L)$ and we denote by $\rho(\tau)$ the automorphism of $(X,L)$
corresponding to $\tau\in T_{c}.$ Writing $\tau:=(\tau_{1},...,\tau_{m})$
we identify $T_{c}$ with $\C^{*m}$ in the standard way and set $\tau_{j}=e^{t_{j}+i\theta_{j}},$
where $t_{j}$ are real coordinates on $(J$ times) the Lie algebra
$\mbox{Lie \ensuremath{(T),}}$ identified with $\R^{m},$ which we
equip with its standard partial order relation $\leq.$ Similarly,
we identify the dual $\mbox{Lie \ensuremath{(T)}}^{*}$ with $\R^{m}$
with the corresponding dual real coordinates, which will be denote
by $\lambda:=(\lambda_{1},...,\lambda_{m}).$ For $t:=(t_{1},...,t_{n})$
and $\phi\in PSH(X,L)$ we set 
\[
\phi_{t}:=\rho(\tau)^{*}\phi,
\]
 which gives a family of metrics $\phi_{t}$ parametrized by $\R^{n}.$
Next, for $\phi$ smooth we define the corresponding moment map $m_{\phi}$by
differentiating wrt $t:$
\[
m_{\phi}(x):=\nabla_{t}\phi_{t}(x)_{|t=0}:\,\,\,\,\,\, X\rightarrow P\subset\R^{m},
\]
 where $P$ is the defined as the image 
\[
P:=m_{\phi}(X)(\subset\R^{m})
\]
 which is well-known to be independent of $\phi$ (and compact, since
it is the image of the compact manifold $X$ under a continuous map).
In fact, $P$ is convex and coincides with the support of the corresponding
\emph{Duistermaat-Heckman measure} 
\[
\nu:=(m_{\phi})_{*}\frac{(\omega^{\phi}){}^{n}}{n!V}
\]
on $\R^{m},$ which is absolutely continuous with respect to Lesbegue
measure and independent of $\phi$ (see \cite{d-h} for the case when
$L$ is ample and \cite{a-b} for the semi-positive case); the independence
of $\nu$ wrt $\phi$ will also follow from Proposition \ref{prop:conv of spec meas}
(see also Remark \ref{ref:independence}). For $\lambda\in P$ and
$\phi$ in $PSH(X,L)^{T}$ given we define the following singular
metric on $L:$ 
\begin{equation}
P_{\lambda}\phi:=\psi_{\lambda}:=\inf_{t\geq0}\left(\phi_{t}-\left\langle t,\lambda\right\rangle \right).\label{eq:def of lower env metric}
\end{equation}

\begin{lem}
\label{lem:Plambdaphi is local bounded on}For any $\lambda$ such
that $\lambda$ is in the interior of $P$ the metric $\psi_{\lambda}$
is in $PSH(X,L)$ and if $\phi$ is locally bounded then $\psi_{\lambda}$
is locally bounded on a Zariski open subset $\Omega$ of $X.$ \end{lem}
\begin{proof}
By Kiselman's minimum principle the metric $\psi_{\lambda}$ is psh
(note that $\psi_{\lambda}$ is automatically upper-semi continuities,
since it is the infimum of continuous metrics). Next, observe that
there exists an integer vector $\lambda^{(k)}$ in $P$ such that
$\lambda^{(k)}/k\geq\lambda$ and a holomorphic section $s_{k}$ in
$H^{0}(X,kL)$ (for some large $k)$ satisfying $\tau^{*}s_{k}=e^{\left\langle (i\theta+t),\lambda^{(k)}\right\rangle }s_{k}$
(as follows from Proposition \ref{prop:conv of spec meas}). Set $\phi_{k}:=\frac{1}{k}\log|s_{k}|^{2}-C_{k},$
where the constant has been chosen so that $\phi\geq\phi_{k}.$ In
particular $\tau^{*}\phi\geq\tau^{*}\phi_{k}$ and hence $\phi_{t}\geq\left\langle t,\lambda^{(k)}/k\right\rangle -\left\langle t,\lambda\right\rangle +\log|s_{k}|^{2}.$
Hence, letting $X-\Omega$ be the zero set $\{s_{k}=0\}$ and taking
the inf over all $t\geq0$ reveals that $\phi_{t}(x)$ is locally
bounded on $\Omega.$ 
\end{proof}
The operator $P_{\lambda}$ on $PSH(X,L)^{T}$ defined by formula
\ref{eq:def of lower env metric} has the following basic properties,
which follow immediately from its variational definition:
\begin{lem}
\label{lem:basic prop of lower env metric}In general $P_{\lambda}\phi\leq\phi$
and 
\begin{itemize}
\item If $\phi\leq\phi'$ then $P_{\lambda}\phi\leq P_{\lambda}\phi$ and
$P_{\lambda}(\phi+C)=P_{\lambda}\phi+C$ for any constant $C\in\R.$ 
\end{itemize}
In particular, if $\phi-\phi'$ is bounded then so is $P_{\lambda}(\phi)-P_{\lambda}\phi'.$
Moreover, if $\phi_{j}$ decreases to $\phi$ the metric $P_{\lambda}\phi_{j}$
decreases to $P_{\lambda}\phi.$ 
\end{lem}
The relation to the moment map is given by the following 
\begin{lem}
Assume that $\phi$ is in $\mathcal{H}(X,L)^{T}.$ Then
\[
\left\{ m_{\phi}\geq\lambda\right\} =\left\{ P_{\lambda}\phi=\phi\right\} 
\]
and setting $\chi_{\lambda}(p):=1_{\{p\geq\lambda\}}$ on $\R^{m}$
the following identity holds: 

\begin{equation}
MA_{\chi_{\lambda}}(\phi)=MA(P_{\lambda}(\phi))\label{eq:relation for ma of chi in lemma}
\end{equation}
for almost any $\lambda$ in $P.$\end{lem}
\begin{proof}
The first relation is an immediate consequence of the convexity of
$t\mapsto\phi_{t}$ on $[0,\infty[^{n}$ (and the fact that $\phi_{0}=\phi).$
Next, we observe that $1_{\{\psi_{\lambda}=\phi\}}MA(\phi)\leq1_{\{\psi_{\lambda}=\phi\}}MA(\psi_{\lambda}),$
as follows immediately from the fact that the set $\{\psi_{\lambda}=\phi\}$
is, for a.e. $\lambda,$ the closure of an open domain of $X$ (by
Sard's theorem). To prove the reversed inequality we apply the comparison
principle for the MA-operator to $\psi_{\lambda}$ and $\phi-\epsilon$
(which indeed applies since $\psi_{\lambda}$ is more singular than
$\phi;$ see Remark 2.4 in \cite{begz}) to get 
\[
\int_{\{\psi_{\lambda}>\phi-\epsilon\}}MA(\phi)\geq\int_{\{\psi_{\lambda}>\phi-\epsilon\}}MA(\psi_{\lambda}).
\]
 Letting $\epsilon\rightarrow0$ thus gives $1_{\{\psi_{\lambda}=\phi\}}MA(\phi)\geq1_{\{\psi_{\lambda}=\phi\}}MA(\psi_{\lambda}).$
Finally, the relation \ref{eq:relation for ma of chi in lemma} follows
from the fact that $MA(\psi_{\lambda})=0$ on the complement of the
set where $\psi_{\lambda}=\phi,$ i.e.ion the open set where $\psi_{\lambda}<\phi.$
This follows from general maximality properties of envelopes (compare
\cite{b-b}), but in the present setting it also follows directly
from the fact that $\psi_{\lambda}$ harmonic along the orbits of
$T_{c}$ in the open subset $\psi_{\lambda}<\phi$ or from the proof
of Proposition \ref{prop:conv of g-bergman}.
\end{proof}

\subsubsection{\label{sub:The g-Monge-Amp=0000E8re-measure}The $g-$Monge-Ampère
measure}

Given a continuous function $g$ on $P$ which will be assumed to
be continuous (and usually normalized so that $g\nu$ is a probability
measure) we define, for $\phi$ in $\mathcal{H}(X,L)^{T},$ the measure

\begin{equation}
MA_{g}(\phi):=MA(\phi)g(m_{\phi}),\label{eq:def of mag for phi smooth}
\end{equation}
on $X,$ which will be referred to as the\emph{ $g-$Monge-Ampère
measure} (or the \emph{$g-$modified Monge-Ampère measure}). In order
to define $MA_{g}(\phi)$ for $\phi$ merely locally bounded first
consider the (non-continuous) case when $g$ is of the form $g=\chi_{\lambda}:=1_{\{p\geq\lambda\}}(p)$
and define $MA_{\chi_{\lambda}}(\phi)$ by the relation \ref{eq:relation for ma of chi in lemma}.
Next, by imposing linearity wrt $g$ this defines $MA_{g}(\phi)$
for any step function $g$ (where the ``step'' is a cube in $\R^{m}).$
In the case of a continuous $g$ we set

\[
MA_{g}(\phi):=\lim_{j\rightarrow\infty}MA_{g_{j}}(\phi)
\]
where $g_{j}$ is any sequence of step functions converging uniformly
to $g$ on $P.$ Finally, for $\phi$ a general (possibly singular)
metric in $PSH(X,L)$ we set 
\begin{equation}
MA_{g}(\phi):=\lim_{k\rightarrow\infty}1_{\{\phi^{(k)}>\phi-k_{0}\}}MA_{g}(\phi^{(k)}),\,\,\,\,\,\phi^{(k)}:=\max\{\phi,\phi_{0}-k\}\label{eq:def of nonpluripol g ma measure}
\end{equation}

\begin{thm}
\label{thm:g-ma is well-def and cont}The measure $MA_{g}(\phi)$
is a well-defined measure on $X,$ not charging pluripolar subset
and 

\begin{equation}
\int_{X}MA_{g}(\phi)=\int_{P}g\nu\label{eq:total mass of g ma in theorem}
\end{equation}
 if $\phi$ has full Monge-Ampère mass (with an inequality $\leq$
for $\phi$ a general metric in $PSH(X,L)^{T}).$ Moreover, if $\phi_{j}$
is a sequence of metrics with full Monge-Ampère mass decreasing to
$\phi$ with full Monge-Ampère mass, then 
\[
MA_{g}(\phi_{j})\rightarrow MA_{g}(\phi)
\]
in the weak topology of measures on $X$ and more generally: 
\[
(\psi_{j}-\phi_{0})MA_{g}(\phi_{j})\rightarrow(\psi-\phi_{0})MA_{g}(\phi)
\]
 if $\psi_{j}$ decreases to a locally bounded metric $\psi$ in $PSH(X,L)^{T}.$
In particular, the uniqueness statement in Theorem \ref{thm:g-ma intro}
also holds.\end{thm}
\begin{proof}
\emph{Step one: the case of $\phi$ locally bounded. }

By the very definition of $MA_{g},$ for $\phi$ in $\mathcal{H}(X,L)^{T}$
the equation \ref{eq:total mass of g ma in theorem} holds for all
functions $g$ iff $(m_{\phi})_{*}MA(\phi)=\nu,$ where $\nu$ is
the canonical D-H measure attached to $T.$ As pointed out above the
latter push-forward relation indeed holds for  smooth (a proof in
the spirit of the present paper is explained in the remark below which
also applies to metrics $\phi$ which are merely locally bounded).
In particular, if $g_{j}$ and $g'_{j}$ are two sequences of step
functions converging to the same continuous function $g,$ then $\int|MA_{g_{j}}(\phi)-MA_{g'_{j}}(\phi)|$
tends to zero when $j\rightarrow\infty$ which shows that $MA_{g}(\phi)$
is well-defined for $\phi$ locally bounded. To prove the continuity
statement when $\phi$ is locally bounded first observe that the continuity
with respect to decreasing sequences holds for $g$mof the form $g=\chi_{\lambda}$
and hence, by linearity as long as $g$ is a step function. Indeed,
if $\phi_{j}$ decreases to $\phi,$ then $P_{\lambda}\phi_{j}$ decreases
to $P_{\lambda}\phi$ and $P_{\lambda}\phi-P_{\lambda}\phi_{j}$ is
bounded (by Lemma \ref{lem:basic prop of lower env metric}). But
then the desired continuity follows immediately from Prop \ref{prop:cont of ma with sing}.
To handle the case of a general $g$ we take two sequences of step
functions $g_{k}^{\pm}$ satisfying $g_{k}^{-}\leq g\leq g_{k}^{+}$
and converging uniformly to $g.$ Fixing a smooth positive function
$u$ on $X$ we then have 
\[
\int MA_{g_{k}^{-}}(\phi_{j})u\leq\int MA_{g}(\phi_{j})u\leq\int MA_{g_{k}^{+}}(\phi_{j})u
\]
for $j$ and $k$ fixed. Letting $j\rightarrow\infty$ for $k$ fixed
and using the continuity property of $MA_{g_{k}^{-}}$ established
above thus gives 
\[
\int MA_{g_{k}^{-}}(\phi)u\leq\lim_{j\rightarrow\infty}\int MA_{g}(\phi_{j})u\leq\int MA_{g_{k}^{+}}(\phi)u
\]
Finally, letting $k\rightarrow\infty$ and using the definition of
$MA_{g}(\phi)$ reveals that the sequences $\int MA_{g_{k}^{-}}(\phi)u$
and $\int MA_{g_{k}^{+}}(\phi)u$ both converge to $\int MA_{g}(\phi)u,$
which concludes the proof in the locally bounded case. 

\emph{Step two: the case of $\phi$ with full Monge-Ampère mass}

By the definition  of $MA_{g}(\phi)$ and the previous case it will
be enough to show that, for any continuous function $h$ on $X,$
\[
\int_{X}h\left(MA_{g}(\psi^{(k)})-MA_{g}(\psi)\right)
\]
tends to zero uniformly for all $\psi\geq\phi.$ But the absolute
value of the integral above may be estimated by a constant $C$ times
\[
\int_{\{\psi\leq\phi_{0}-k\}}\left(MA(\psi^{(k)})+MA(\psi)\right),
\]
 where the constant $C$ only depends $\sup_{X}|h|$ and $\sup_{P}|g|.$
But, as shown in the proof of Theorem 2.17 in \cite{begz} the latter
integral tends to zero if $\phi$ has full Monge-Ampère mass, which
thus proves the general continuity statement. In particular, regularizing
$\phi$ and applying the previous step also gives equality in equation
\ref{eq:total mass of g ma in theorem} for $\phi$ of full Monge-Ampère
mass (and inequality in general). The proof of the last convergence
statement is proved in a similar manner.Finally, the fact that the
limit \ref{eq:def of nonpluripol g ma measure} is well-defined for
any $\phi\in PSH(X,L)^{T}$ follows from the $T-$plurifine locality
of $MA_{g}$ (acting on locally bounded $\phi;$ see section \ref{sub:Fine-locality-and}
below). Indeed, setting $E_{k}:=\{\phi^{(k)}>\phi-k_{0}\}$ gives,
by locality, that $1_{\{\phi^{(k)}>\phi-k_{0}\}}MA_{g}(\phi^{(k)})=1_{E_{k}}MA_{g}(\phi_{|E_{k}})$
which is increasing in $k,$ with uniformly bounded mass, which ensure
the existence of the limit. Finally, since $MA(\phi)$ does not charge
pluripolar subsets so doesn't $MA_{g}(\phi)$ and this argument also
proves the uniqueness statement in Theorem \ref{thm:g-ma intro} (since
$\{\phi=\infty\}$ is pluripolar).\end{proof}
\begin{rem}
\label{ref:independence}The arguments above can be used to give a
(non-standard) proof of the independence of the measure $\nu.$ The
point is that if $|\phi_{1}-\phi_{2}|\leq C$ then $|P_{\lambda}(\phi_{1})-P_{\lambda}(\phi_{2})|\leq C$
(compare Lemma \ref{lem:basic prop of lower env metric}) and hence,
by Prop \ref{prop:cont of ma with sing}, the total Monge-Ampère mass
$\int MA(P_{\lambda}\phi)$ is the same for any locally bounded metric
$\phi.$ But then it follows (just as above) that, for any $g,$ $\int MA_{g}(\phi)(=\int_{P}g\nu_{\phi})$
is independent of $\phi,$mas desired.
\end{rem}
An immediate consequence of the previous theorem, combined with the
regularization result in Theorem \ref{thm:(regularization)--Any}
together with the compactness of $P,$ is the following
\begin{cor}
\label{cor:comp of g ma and ma}The following inequalities hold: 
\[
\inf_{P}g_{X}MA(\phi)\leq MA_{g}(\phi)\leq\sup_{P}g_{X}MA(\phi)
\]

\end{cor}

\subsubsection{\label{sub:Fine-locality-and}Fine locality and consequences}

By definition, the $T-$plurifine topology on $X$ is the coursest
topology making all sets of the form $O:=\{\phi<\psi\}$ open, for
$\phi$ and $\psi$ metrics in $PSH(X,L)^{T}$ (which by a max construction
may be assumed locally bounded). 
\begin{prop}
The operator $\phi\mapsto MA_{g}(\phi)$ on $PSH(X,L)^{T}$ is local
with respect to the $T-$mlurifine topology, i.e. if i.e. if $\phi=\phi'$
on a $T-$plurifine open set $O$ then $1_{O}MA_{g}(\phi)=1_{O}MA_{g}(\phi').$
In particular, 
\begin{equation}
1_{\{\phi>\psi\}}MA_{g}(\phi)=1_{\{\phi>\psi\}}MA_{g}(\max(\phi,\psi))\label{eq:locality for ma of max}
\end{equation}
\end{prop}
\begin{proof}
By the definition of $MA_{g}$ we may as well assume that $\phi$
and $\phi'$ are locally bounded. Regularizing $\phi$ and $\phi'$
and invoking the convergence result in Theorem \ref{thm:g-ma is well-def and cont},
the proof is reduced - precisely as in the classical case when $g=1$
\cite{b-t2} - to the case when $\phi$ and $\phi'$ are smooth and
$O$ is open with respect to the standard topology. But then the locality
in question follows trivially from formula \ref{eq:def of mag for phi smooth}. 
\end{proof}
Just as in the case $g=1$ formula \ref{eq:locality for ma of max}
implies that the comparison principle holds (also using that the total
mass of $MA_{g}(\phi)$ is independent of $\phi$ when $\phi$ has
maximal Monge-Ampère mass). 
\begin{prop}
\label{prop:(the-comparison-principle).}(the comparison principle).
Assume that $\phi$ and $\psi$ are in $PSH(X,L)$ and of full Monge-Ampère
mass. Then

\[
\int_{\{\psi>\phi\}}MA_{g}(\psi)\leq\int_{\{\psi>\phi\}}MA_{g}(\phi)
\]

\end{prop}
Another useful consequence of the formula \ref{eq:locality for ma of max}
is the following inequality (compare \cite{begz} for the case $g=1$): 
\begin{prop}
\label{prop:max of max}Assume that $\phi$ and $\psi$ are in $PSH(X,L)^{T}$cmnd
that there exists a positive measure $\mu$ such that $MA_{g}(\phi)\geq\mu$
and $MA_{g}(\psi)\geq\mu.$ Then

\[
MA_{g}(\max(\phi,\psi))\geq\mu
\]

\end{prop}

\subsection{\label{sub:Interlude-I:-the vector}Vector fields generating torus
actions on singular varieties }

In this section we assume that $L$ is the pull-back $L$ of an ample
line bundle $A$ on a (possible singular) projective variety $Y.$
Let $V$ be a complex holomorphic vector field on $X$ (i.e. $V\in H^{0}(X,T^{1,0}X))$
with a fixed holomorphic lift to $L.$ Denote by $\mathcal{H}(X,L)^{V}$
the space of all smooth metrics on $L$ with semi-positive curvature
form such that $\phi$ is invariant under the flow $\mbox{\ensuremath{\exp}(\ensuremath{t\mbox{Im}V))}}$
defined by the imaginary part of $V$ (equivalently: $L_{\mbox{Im}V}\phi=0,$
where $L_{\mbox{Im}V}$ denotes the Lie derivative along the imaginary
part of $V)$ and define $PSH(X,L)^{V}$ etc in a similar manner. 
\begin{lem}
\label{lem:vector field torus}Let $V$ be a holomorphic vector field
on $X$ with admits a holomorphic lift to a line bundle $L\rightarrow X,$
where $L$ is assumed to be the pull-back of an ample line bundle
$A$ on a (possible singular) projective variety $Y.$ If there exists
a metric $\phi$ on $L$ with positive curvature current which is
invariant under the flow of the imaginary part of $V$ and such that
$\phi$ is locally bounded, then there exists a complex torus $T_{c}$
acting holomorphically on $(X,L)$ such that the imaginary (as well
as the real) part of $V$ may be identified (in the standard way)
with an element $\xi$ in the Lie algebra of the corresponding real
torus $T\subset T_{c}$ and $PSH(X,L)^{V}=PSH(X,L)^{T}$. Moreover,
in the in the case when $A=-K_{Y}$ for $Y$ a Fano variety with log
terminal singularities the boundedness assumption on $\phi$ may be
replaced by the assumption that $\phi$ has full Monge-Ampère mass.\end{lem}
\begin{proof}
Let $K$ be the subgroup of $\mbox{Aut \ensuremath{(X,L)_{0}}}$ fixing
the metric $\phi$ and denote by $G$ the one-parameter subgroup $G$
of $K$ defined by the flow of the imaginary part of $V.$ First observe
that $K$ is a compact Lie group. Indeed,  induces a $K-$invariant
Hermitian norm on $H^{0}(X,kL).$ In the case that $\phi$ is locally
bounded the norm may be defined by $\left\Vert s_{k}\right\Vert ^{2}:=\int|s_{k}|^{2}e^{-k\phi}MA(\phi)$
and in the case $A=-K_{Y}$ we can replace $MA(\phi)$ with the pull-back
to $X$ of the measure $\mu_{\phi}$ on $Y$ with local density $e^{-\phi}$
defined by the metric $\phi$ (compare formula \ref{eq:def of mu phi}).
Since $\phi$ is assumed to have full MA-mass it has no Lelong numbers
and hence $\left\Vert s_{k}\right\Vert ^{2}:=\int_{Y}|s|_{k}^{2}e^{-(k+1)\phi}<\infty$
for any $k>0$ (see the appendix in \cite{bbegz}). The $K-$invariant
norm induces a Kodaira embedding of $X$ into $\P H^{0}(X,kL)$ and
hence $K$ may be identified with the subgroup of $U(N+1,\C)$ preserving
the image of $X,$ which is clearly a compact Lie group Next, denote
by $\bar{G}$ the topological closure of $G$ in $K,$ which defines
a topological compact and connected abelian subgroup of the compact
Lie group $K.$ But it is well-known that any such subgroup is a Lie
subgroup and hence it follows from the standard classification of
such Lie groups that $\bar{G}$ is a real torus. Finally, by general
principles (analytic continuation) any holomorphic action of a real
torus on $(X,L)$ can be extended holomorphically to give an action
of the corresponding complex torus, which concludes the proof. 
\end{proof}
Now, given a metric $\phi$ in $\mathcal{H}(X,L)^{V}$ we define the
following smooth function on $X$ 
\[
f_{\phi}:=\left\langle m_{\phi},\zeta\right\rangle ,
\]
 which in (in the case when $\omega_{\phi}>0)$ is a \emph{Hamiltonian
function }for the $\omega^{\phi}-$symplectic real vector field $\mbox{Im\ensuremath{V}and}$

\[
\omega^{\phi}(\mbox{Im}V,\cdot)=dh_{\phi},
\]
or equivalently: 
\[
\nabla^{\phi}h_{\phi}=\mbox{Re}V
\]
where $\nabla^{\phi}$ is the gradient defined with respect to $\omega^{\phi}$
(the previous relation holds where it is well-defined, i.e. where
$\omega^{\phi}>0$). In particular, it follows from the compactness
of $P\subset\R^{m}$ that there exists a universal constant $C$ such
that 
\[
\sup_{X}|f_{\phi}|\leq C
\]
(in the case when $L$ is ample this was shown by a different argument
in \cite{z}). Now given a continuous function $G(v)$ on $\R$ we
set $g(\lambda):=G(\left\langle \lambda,\zeta\right\rangle )$ and
obtain, by the construction in section \ref{sub:The g-Monge-Amp=0000E8re-measure},
a modified Monge-Ampère measure
\[
MA_{(G,V)}(\phi):=MA_{g}(\phi),
\]
 defined for any metric $\phi$ in $PSH(X,L)^{V}$ with the property
that 
\[
MA_{(G,V)}(\phi):=MA(\phi)G(h_{\phi})
\]
 for $\phi$ smooth. In particular, in the setting of Kähler-Ricci
solitons considered in section \ref{sec:K=0000E4hler-Ricci-solitons}
we will take $G(v)=e^{v}/C$ and write $g_{V}$ for the corresponding
function on the Lie algebra of $T.$ Here the normalizing constant
$C$ is chosen so that $g_{V}\nu$ is a probability measure.

\subsubsection{The case of a singular Fano variety}

By normality it is enough to define the holomorphic vector field $V,$
appearing in Lemma \ref{lem:vector field torus}, on the regular part
of a singular normal variety $Y.$ In particular, if $Y$ is a normal
Fano variety, i.e. $-K_{Y}$ is ample (see section \ref{sec:K=0000E4hler-Ricci-solitons}),
then any holomorphic vector field $V_{0}$ defined on the regular
locus $Y_{0}$ of $Y$ admits a canonical lift to $-K_{Y_{0}}$ (since
$V_{0}$ naturally acts on the tangent bundle of $Y_{0}).$ Hence,
if $\pi:\, X\rightarrow Y$ is a resolution such that $\pi$ is an
isomorphism over $Y_{0}$ then, by normality, the vector field $V_{0}$
admits a unique extension $V_{X}$ to $X$ such that $V_{X}$ lifts
to the line bundle $L:=\pi^{*}(-K_{X}).$ As a consequence the corresponding
$g-$Monge-Ampère measure $MA_{g_{V}}(\phi)$ can be defined on the
singular variety $Y,$ by passing to the resolution $X.$ The measure
$MA_{g_{V}}(\phi)$ thus defined is independent of the resolution
$X,$ since the $g-$Monge-Ampère measure on $X$ does not charge
pluripolar sets and in particular it does not charge the exceptional
locus of $\pi.$

\subsection{\label{sub:Energy-type-functionals and proj}Energy type functionals
and psh projections }

As is well-known the classical Monge-Ampère measure $MA$ viewed as
a one-form on the convex space $\mathcal{H}(X,L)$ is closed and hence
exact. In particular, $MA$ admits a primitive, i.e. a functional
$\mathcal{E}$ on $\mathcal{H}(X,L)$ such that $d\mathcal{E}{}_{|\phi}=MA(\phi),$
where $d\mathcal{E}$ denotes the differential of $\mathcal{E},$
which is uniquely determined by the normalization condition $\mathcal{E}(\phi_{0})=0.$
Integrating along affine segments in $\mathcal{H}(X,L)$ one arrives
at the following explicit energy type expression: 
\begin{equation}
\mathcal{E}(\phi):=\frac{1}{(n+1)c_{1}(L)^{n}}\int_{X}(\phi-\phi_{0})\sum_{j=0}^{n}(dd^{c}\phi)^{n-j}\wedge(dd^{c}\phi_{0})^{j}\label{eq:energy e beautiful explicit}
\end{equation}
Following \cite{begz,bbgz} the functional $\mathcal{E}(\phi)$ admits
a unique increasing and usc extension to all of $PSH(X,L)$ given
by 
\[
\mathcal{E}(\phi):=\inf_{\psi\geq\phi}\mathcal{E}(\psi),
\]
 where $\psi$ ranges over all elements of $\mathcal{H}(X,L)$ bounded
from below by $\phi.$ A metric $\phi$ is said to have finite (pluricomplex)
energy if $\mathcal{E}(\phi)>-\infty$ and the space of all finite
energy metrics is denoted by $\mathcal{E}^{1}(X,L).$ As shown in
\cite{begz} any finite energy metric has full Monge-Ampère mass,
i.e. 
\[
\mathcal{H}(X,L)\subset\mathcal{E}^{1}(X,L)\subset\mathcal{E}(X,L)
\]
Let us also recall the definition of the psh projection operator $P$
from the space $C(X,L)$ of all continuous metrics on $L$ to the
space of all continuous and psh (i.e. positively curved) metrics defined
by the following point-wise upper envelope: 
\[
P\phi:=\sup\left\{ \psi:\,\,\psi\leq\phi\,\,\,\,\psi\in PSH(X,L)\right\} 
\]
Next, turning to the present torus setting we start by observing that
the one-form $MA_{g}$ on $\mathcal{H}(X,L)^{T}$ defined by the $g-$Monge-Ampère
measure is also closed and hence exact (this observation is essentially
due to Zhu \cite{z}):
\begin{lem}
There exists a functional $\mathcal{E}_{g}$man $\mathcal{H}(X,L)^{T}$
such that $(d\mathcal{E}_{g})_{|\phi}=MA_{g}(\phi).$ The functional
$\mathcal{E}_{g}$ is uniquely determined by the normalization condition
$\mathcal{E}_{g}(\phi_{0})=0.$\end{lem}
\begin{proof}
This was essentially shown by Zhu \cite{z} in the vector field setting
for $g(p)$ the exponential function (compare section \ref{sub:Interlude-I:-the vector})
and the proof in the general case is essentially the same (see also
\cite{bern1}). Alternatively, the existence of the functional $\mathcal{E}_{g}$
also follows from the existence of its quantized version $\mathcal{L}_{g}$
(see Prop \ref{prop:conv of l-functional}).
\end{proof}
Proceeding as before we may now extend $\mathcal{E}_{g}$ to all of
$PSH(X,L)^{T}$ by setting 
\[
\mathcal{E}_{g}(\phi):=\inf_{\psi\geq\phi}\mathcal{E}_{g}(\psi),
\]
where $\psi$ ranges over all elements of $\mathcal{H}(X,L)^{T}$
bounded from below by $\phi.$ It follows immediately from the definition
of $\mathcal{E}_{g}(\phi)$ and Cor \ref{cor:comp of g ma and ma}
that there exists a constant $C$ such that 
\[
\mathcal{E}_{g}(\phi)\geq C\mathcal{E}(\phi)
\]
 under the normalizing condition $\sup_{X}(\phi-\phi_{0})=0.$ In
particular, $\mathcal{E}_{g}(\phi)$ is finite for any metric $\phi$
of finite energy. 
\begin{prop}
\label{prop:cont prop of g-energy}The functional $\mathcal{E}_{g}(\phi)$
on $PSH(X,L)^{T}$ has the following properties:
\begin{itemize}
\item $\mathcal{E}_{g}$ is increasing and satisfies $\mathcal{E}_{g}(\phi+c)=\mathcal{E}_{g}(\phi)+c$
for any constant $c\in\R$ iff $gd\nu$ is a probability measure. 
\item $\mathcal{E}_{g}$ is upper semi-continuous with respect to the $L^{1}-$topology
on $PSH(X,L)^{T}$
\item $\mathcal{E}_{g}$ is continuous along decreasing sequences in $PSH(X,L)^{T}$
\end{itemize}
\end{prop}
\begin{proof}
The first point follows immediately from from the fact that $MA_{g}(\phi)$
is a probability measure when $\phi\in\mathcal{H}(X,L)^{T}.$ In the
case when $\phi$ is locally bounded the continuity of $\mathcal{E}_{g}$
wrt decreasing sequences follows from the convergence in Theorem \ref{thm:g-ma is well-def and cont},
by writing $\mathcal{E}_{g}(\phi)=\int_{X\times[0,1]}(\phi-\phi_{0})MA_{g}(t\phi+(1-t)\phi_{0})dt.$
By a simple max construction this implies the general upper-semicontinuity
in the second point and hence, since $\mathcal{E}_{g}$ is increasing,
it also implies the general continuity wrt decreasing sequences in
the third point (just as in the proof of Prop 2.10 in \cite{begz}
)
\end{proof}
We also have the following generalization of the differentiability
theorem in \cite{b-b} to the torus setting, which plays a key role
in the variational approach to complex Monge-Ampère equations \cite{bbgz}:
\begin{prop}
\label{prop:diff of The-composed-function }The composed functional
$\mathcal{E}_{g}\circ P$ is Gateaux differentiable on the affine
space $C(X,L)^{T}$ of all $T-$ invariant continuous metrics on $L$
and 
\begin{equation}
d(\mathcal{E}_{g}\circ P)_{|\phi}=MA_{g}(P\phi)\label{eq:diff of composed energy proj}
\end{equation}
More generally, if $\phi$ is in $\mathcal{E}^{1}(X,L)^{T}$ and $u$
is a continuous function on $X,$ then 
\[
\frac{d\left(\mathcal{E}_{g}(P(\phi+tu)\right)}{dt}_{|t=0}=\int MA_{g}(P\phi)u
\]
\end{prop}
\begin{proof}
First observe that the following ``orthogonality relation'' holds
for $\phi$ a metric in $\mathcal{C}(X,L)^{T}:$
\[
\int MA_{g}(P\phi)(\phi-P\phi)=0
\]
Indeed, for $g=1$ this is well-known (compare \cite{b-b}) and the
general case then follows immediately from the upper bound in Cor
\ref{cor:comp of g ma and ma}. Combining the orthogonality relation
above with the comparison principle in Prop \ref{prop:(the-comparison-principle).}
the proof of formula \ref{eq:diff of composed energy proj} then follows
using the corresponding comparison principle (Prop \ref{prop:(the-comparison-principle).}),
just as in \cite{b-b}. Finally, the proof of the last formula in
the proposition is reduced to the previous case by approximation,
just as in the proof of Lemma 4.2 in \cite{bbgz}.
\end{proof}

\subsection{Energy along $T-$invariant bounded geodesics }

As noted by Mabuchi, Semmes and Donaldson, independently, there is
a natural Riemannian metric on the infinite dimensional space $\mathcal{H}_{\infty}$
of all smooth positively curved metrics on a given line bundle $L,$
such that corresponding geodesics $\phi_{t}$ - when then they exist
- are solutions to a complex Monge-Ampère equation. Here we will only
use the pluripotential version of this setup where a notion of ``weak
geodesic'' or more precisely ``bounded geodesic'' is defined directly
without any reference to a Riemannian structure \cite{bbgz,b-d,bern1}:
given two elements $\phi_{0}$ and $\phi_{1}$ in $PSH(X,L)\cap L_{loc}^{\infty}$
there is a unique ``bounded geodesic'' $\phi_{t}$ defining a curve
in $PSH(X,L)\cap L_{loc}^{\infty}$ connecting $\phi_{0}$ and $\phi_{1}.$
The curve $\phi_{t}$ may be defined as follows: first we set $t:=\log|\tau|,$
for $\tau\in\C$ and identify $\phi_{t}$ with an $S^{1}-$metric
$\Phi$ on $\pi^{*}L\rightarrow X\times A,$ where $A$ is the annulus
$\{\tau\in\C:\,\,0\leq\log|\tau|\leq1\}$ in $\C$ equipped with its
standard $S^{1}-$action and $\pi$ is the standard projection from
$X\times A$ to $X.$ The restriction $\Phi_{|\partial(X\times A)}$
to the boundary of $X\times A$ is determined by the given metrics
$\phi_{0}$ and $\phi_{1}$ and the extension to all of $X\times A$
of $\Phi$ is given by the following envelope:

\[
\Phi:=\sup_{\Psi}\left\{ \Psi:\,\,\,\,\Psi_{|\partial(X\times A)}\leq\Phi_{|\partial(X\times A)}\right\} 
\]
where $\Psi$ ranges of all $S^{1}-$invariant bounded metrics on
$\pi^{*}L\rightarrow X\times A$ with positive curvature current (the
corresponding curve $\psi_{t}$ will be called a bounded\emph{ subgeodesic}).
As observed in \cite{bern1} a simple barrier argument reveals that
$\Phi$ is continuous at the boundary in the sense that $\phi_{t}\rightarrow\phi_{0}$
uniformly as $t\rightarrow0$ and similarly for $t=1.$ Indeed, this
follows directly from the fact that the continuous metric 
\begin{equation}
\Psi:=\max\{\Phi_{0}-A\log|\tau|,\Phi_{1}-A(1-\log|\tau|)\}\label{eq:barrier}
\end{equation}
 is, for $A$ sufficiently large a candidate for the sup defining
 Moreover, by standard properties of free envelopes 
\[
(dd^{c}\Phi)^{n+1}=0
\]
on the interior of $X\times A.$ In other words, the extension $\Phi$
is a solution of the corresponding Dirichlet problem for the Monge-Ampère
operator on the domain $X\times A$ (as shown in \cite{b-d} $\Phi$
is continuous on all of $X\times A$ if the given boundary data is).
It follows immediately from the envelope construction above that if
$\phi_{0}$ and $\phi_{1}$ are $T-$invariant, then so is $\phi_{t}$
for any $t\in[0,1].$ 
\begin{prop}
\label{prop:energy along geodesics}Let $\phi_{t}$ be a curve in
$\mathcal{E}^{1}(X,L)^{T}$ defined for $t\in[0,1]$ and set $f(t):=\mathcal{E}_{g}(\phi_{t})$
\begin{itemize}
\item If $\phi_{t}$ is a bounded subgeodesic, then the function $f(t)$
is convex on $]0,1[.$ 
\item If $\phi_{t}$ is a bounded geodesic then $f(t)$ is affine on $[0,1]$
(and continuous up to the boundary)
\item If $\phi_{t}$ is affine wrt $t,$ then $f(t)$ is concave, i.e. the
functional $\mathcal{E}_{g}$ is concave on the space $\mathcal{E}^{1}(X,L)^{T}$
equipped with its affine structure.
\end{itemize}
\end{prop}
\begin{proof}
A direct calculation reveals that, if the corresponding metric $\Phi$
on $\pi^{*}L$ is smooth, then the function $F(\tau):=\mathcal{E}_{g}(\Phi(\cdot,\tau))$
on $A$ satisfies 
\begin{equation}
dd^{c}F=\int_{X}MA(\Phi)g(m_{\Phi(\cdot,\tau)}),\label{eq:ddc of energy}
\end{equation}
 where $\int_{X}\cdot$ denotes the push-forward to $X$ (see the
appendix in \cite{bern1} for the special case when $g$ is a function
on $\R$ of the form $g(v)=e^{av}$ for $a\in\R$; the general case
reduces to this special case by expressing $g$ in terms of its Fourier
transform). Note that the integrand may be written as the $g-$Monge-Ampère
measure $MA_{g}(\Phi)$ on $X\times A,$ defined with respect to the
induced $T\times S^{1}-$action on $(X\times A,\pi^{*}L)$, identifying
$g$ with a function on $\R^{m+1}.$ Let now $U$ be a $S^{1}-$invariant
open set whose closure is contained in the interior of $A.$ Then
it follows from a variant of Theorem \ref{thm:(regularization)--Any}
(or the extension argument in the claim below) that $\phi_{t}$ can
be written as a decreasing sequence $\phi_{t}^{(j)}$ of smooth subgeodesics
defined for $\tau\in U.$ By Prop \ref{prop:cont prop of g-energy}
$\mathcal{E}_{g}(\phi_{t}^{(j)})\rightarrow\mathcal{E}_{g}(\phi_{t}),$
as $j\rightarrow\infty$ and since $\mathcal{E}_{g}(\phi_{t}^{(j)})$
is convex wrt $t$ (by formula \ref{eq:ddc of energy}) so is $\mathcal{E}_{g}(\phi_{t}),$
which proves the first point. 

To prove the second point we will use the following claim which furnishes
approximations with good extension properties: \emph{the restriction
$\Phi_{|U}$ of the given bounded $T\times S^{1}-$invariant metric
on $\pi^{*}L\rightarrow X\times A$ defining the geodesic $\phi_{t}$
extends to a metric $T\times S^{1}-$invariant metric $\tilde{\Phi}$
with positive curvature current on the line bundle $L_{m}:=\pi^{*}L\times\mathcal{O}(m)$
over $X\times\P^{1},$ for some positive integer $m.$ }Accepting
this for the moment we can, by regularization, take smooth $\tilde{\Phi}_{j}$
such metrics on $L_{m}$ decreasing to $\tilde{\Phi}.$ In particular,
by Theorem \ref{thm:g-ma is well-def and cont}, 
\begin{equation}
\int_{X\times\P^{1}}MA_{g}(\tilde{\Phi}_{j})u\rightarrow\int_{X\times\P^{1}}MA_{g}(\tilde{\Phi})u\label{eq:proof of cont of energy along geod}
\end{equation}
 for any smooth function $u$ on $X\times\P^{1}.$ Now, setting $F_{j}(\tau):=\mathcal{E}_{g}(\Phi_{j}(\cdot,\tau))$
we have, as before, that $F_{j}(\tau)\rightarrow F(\tau)$ for any
fixed $\tau$ and since $dd^{c}F_{j}\geq0$ this shows that $dd^{c}F(\tau)\geq0,$
i.e. $f(t)$ is convex on $]0,1[$ as desired. Moreover, since $\Phi$
defines a geodesic we have $MA(\Phi)=0$ in the interior of $X\times A$
and in particular $MA(\tilde{\Phi})=0$ on $X\times U.$ Hence, applying
the convergence result in formula \ref{eq:proof of cont of energy along geod}
to a function $u$ which is the pull-back to $X\times U$ of a smooth
function compactly supported in $U$ reveals that $\int_{A}udd^{c}F_{j}\rightarrow\int_{X\times U}MA_{g}(\tilde{\Phi})u,$
which thus vanishes (using Cor \ref{cor:comp of g ma and ma}). But
then it follows that $f(t)$ is affine on $]0,1[,$ as desired. Finally,
the continuity up to the boundary on $[0,1]$ follows from the corresponding
continuity property of the curve $\phi_{t}$ together with Prop \ref{prop:cont prop of g-energy}.
This concludes the proof of the second point up to the claimed extensionmatoperty
of $\Phi_{|U}$, which is obtained by setting 
\[
\tilde{\Phi}:=\max\{\Psi+B\rho+1/B,\Phi_{|U}\}
\]
 for $B$ sufficiently large, $\Psi$ is the metric defined by formula
\ref{eq:barrier} and $\rho$ is the psh function $\rho:=\max\{\log|\tau|,(\log|\tau|-1\})$
on $X\times A,$ which has the property that $\rho\leq0$ and $\rho<1$
in the interior. Then $\tilde{\Phi}$ is a positively curved extension
of $\Phi$ from $U$ coinciding with $\Psi+B\rho+1/B$ close to the
boundary of $X\times A$ and since $2\log|\tau|$ extends to define
a metric on $\mathcal{O}(1)\rightarrow\P^{1}$ it follows from the
definition \ref{eq:barrier} of $\Psi$ that $\Psi$ extends to a
metric on $L_{m}$ for $m$ a sufficiently large integer (depending
on $A),$ which concludes the proof of the claim.

Finally, to prove the last point it is, by approximation, enough to
consider the case when $\phi_{t}$ is an affine curve in $\mathcal{H}(X,L)^{T}.$
But then the concavity follows immediately from the fact that $MA(\Phi)\leq0$
for $\Phi$ corresponding to an affine curve in $\mathcal{H}(X,L)^{T}$
(and using that $g$ is non-negative). 
\end{proof}

\subsection{\label{sub:Monge-Amp=0000E8re-equations-and}Monge-Ampère equations
and variational principles}

Recall that the pluricomplex energy of a measure $\mu,$ that we shall
denote by $E(\mu),$ may be defined by 
\[
E(\mu):=\sup_{\phi\in PSH(X,L)}-J_{\mu}(\phi)
\]
 where $J_{\mu}$is the following $\R-$invariant functional 
\begin{equation}
J_{\mu}(\phi)=-\mathcal{E}(\phi)+\mathcal{L}_{\mu}(\phi),\,\,\,\,\,\,\,\mathcal{L}_{\mu}(\phi):=\int_{X}(\phi-\phi_{0})\mu\label{eq:def of J and L mu}
\end{equation}
(see \cite{bbgz}). Similarly, in the presence of a torus $T$ we
define, for any given function $g$ on the corresponding moment polytope
$P,$  the $g-$modified energy $E_{g}(\mu)$ be replacing the functional
$\mathcal{E}$ with $\mathcal{E}_{g}$ and the space $PSH(X,L)$ with
its $T-$invariant subspace $PSH(X,L)^{T}.$
\begin{thm}
\label{thm:existence uni ma eq energy}Let $\mu$ be a probability
measure on $X$ which is $T-$invariant and assume that $gd\nu$ is
a probability measure on the moment polytope $P$ such that $1/C\leq g\leq C$
for some positive constant $C.$ Then there exists a finite energy
$T-$invariant solution $\phi$ to the equation 
\begin{equation}
MA_{g}(\phi)=\mu\label{eq:g monge am eq in thm energy}
\end{equation}
iff $\mu$ has finite (pluricomplex) energy. Moreover, the solution
is unique modulo constants and can be characterized as the unique
(mod $\R)$ minimizer of the functional $J_{\mu,g}.$\end{thm}
\begin{proof}
\emph{Existence:} Let $PSH(X,L)_{0}$ be the subspace of all metrics
$\phi$ such that $\sup_{X}(\phi-\phi_{0})=0,$ which is compact in
the $L^{1}-$topology. In particular, $-\mathcal{E}(\phi)\geq0$ on
$PSH(X,L)_{0}.$ Let us first assume that $\mu$ has finite energy,
i.e. $J_{\mu}$ is bounded from below. Then, as shown in \cite{bbgz},
the functional $J_{\mu}$ is even coercive in the sense
\[
J_{\mu}(\phi)\geq A(-\mathcal{E}(\phi))-B
\]
 for some positive constants $A$ and $B.$ In particular, by the
assumption on $g$ the corresponding $g-$modified functional $J_{\mu,g}$
is also coercive. Hence, taking a sequence $\phi_{j}$ in $PSH(X,L)_{0}$
such that $J_{\mu,g}(\phi_{j})\rightarrow\inf J_{\mu,g}$ the coercivity
implies that $-\mathcal{E}(\phi)(\phi_{j})\leq C$ for some constant
$C.$ Accordingly, after perhaps passing to a subsequence, $\phi_{j}$
converges in $L^{1}$ to a metric $\phi_{*}$ in $\{-\mathcal{E}\leq C\}$
(using that $\mathcal{E}$ is usc, i.e. the previous sublevel sets
are compact). Next, as shown in \cite{bbgz} $L_{\mu}$ is finite
and even continuous on the compact sublevel sets $\{-\mathcal{E}\leq C\}$
and since $\mathcal{E}_{g}(\phi)$ is also upper semi-continuous (by
Prop \ref{prop:cont prop of g-energy}) it follows that 
\[
J_{\mu,g}(\phi_{*})=\lim_{j\rightarrow\infty}J_{\mu,g}(\phi_{j}):=\inf J_{\mu,g},
\]
 i.e. $\phi_{*}$ is a minimizer of $J_{\mu,g}.$ Finally, to show
that $\phi_{*}$ satisfies the equation \ref{eq:g monge am eq in thm energy},
we set $f(t):=-\mathcal{E}_{g}(P(\phi+tu))+\mathcal{L}_{\mu}(\phi)$
for a given smooth function $u$ on $X,$ which defines a function
on $\R$ with an absolute minimum at $t=0$ (using that $P$ is decreasing)
and which is differentiable (by Prop \ref{prop:diff of The-composed-function }).
In particular, the derivative of $f$ vanishes at $t=0$ which by
Prop \ref{prop:diff of The-composed-function } implies (since $u$
was arbitrary) that the equation \ref{eq:g monge am eq in thm energy}
holds, as desired.

Conversely, if $\phi$ is a finite energy solution of the equation
\ref{eq:g monge am eq in thm energy}, then it follows from the affine
concavity of $\mathcal{E}_{g}$ (Prop \ref{prop:energy along geodesics})
that $\phi$ is a global minimizer of the convex functional $J_{\mu,g}.$ 

\emph{Uniqueness:} Let now $\psi_{0}$ and $\psi_{1}$ be two finite
energy solutions to the equation \ref{eq:g monge am eq in thm energy}
in the convex space $V$ of all finite energy metrics $\psi$ normalized
so that $L_{\mu}(\psi)=0.$ By the affine concavity of $\mathcal{E}_{g}$
on $V$ this means that $\psi_{i}$ are minimizers of the functional
$-\mathcal{E}_{g}$ on $V.$ Let now $\psi_{t}$ be the affine curve
connecting $\psi_{0}$ and $\psi_{1}$ and set $f_{g}(t):=-\mathcal{E}_{g}(\phi_{t})$
and $f(t):=-\mathcal{E}(\phi_{t}).$ We claim that there exists a
positive constant $a$ such that 
\begin{equation}
\frac{d^{2}}{d^{2}t}f_{g}\geq a\frac{d^{2}}{d^{2}t}f\geq0\label{eq:claimed inequality proof of uniq of finite energy}
\end{equation}
 weakly on $[0,1].$ Accepting this for the moment it follows that
$\frac{d^{2}}{d^{2}t}f_{g}=0$ on $]0,1[$ (since $\frac{d}{dt}f_{g}(0)=\frac{d}{dt}f_{g}(1)=0)$
and hence $f(t)$ is also affine on $[0,1].$ But this means that
$\frac{d}{dt}f(0)=\frac{d}{dt}f(1)=0,$ i.e. $\int(\psi_{0}-\psi_{1})MA(\psi_{0})=\int(\psi_{0}-\psi_{1})MA(\psi_{1})=0.$
In particular, $I(\phi_{0},\phi_{1}):=\int(\psi_{0}-\psi_{1})(MA(\psi_{1})-MA(\psi_{0}))=0.$
But as is well-known this implies by an argument essentially due to
Blocki (see \cite{bbgz,begz} and references therein) that $\psi_{0}-\psi_{1}$
is constant and thus zero, by our normalization assumption. Finally,
the claimed inequality \ref{eq:claimed inequality proof of uniq of finite energy}
follows from formula  \ref{eq:ddc of energy} which gives that for
$\psi_{i}$ smooth 
\[
dd^{c}f_{g}=-\int_{X}MA_{g}(\Psi),\,\,\,\, dd^{c}f=-\int_{X}MA(\Psi)
\]
Moreover, for any affine path $\psi_{t}$ we have (by a direct calculation)
that $MA(\Psi)\leq0$ and hence, by definition, $-MA_{g}(\Psi)\geq-aMA(\Psi)$
proving the inequality \ref{eq:claimed inequality proof of uniq of finite energy}
in the smooth case. The general case then follows by regularization
using that $f_{g}^{(j)}\rightarrow f_{g}$ etc as in the proof of
Prop \ref{prop:energy along geodesics}.
\end{proof}

\subsubsection{Completion of the proof of Theorem \ref{thm:existence uni ma eq energy}}

First observe that since, by assumption $g\geq1/C>0,$ Cor \ref{cor:comp of g ma and ma}
implies that any solution $\phi$ satisfies $MA(\phi)\leq C\mu.$
In particular, if $\mu$ has a density in $L_{loc}^{p}$ for some
$p>1$ then so has $MA(\phi)$ and thus it follows from the generalizations
of Kolodziej's results in \cite{e-g-z} that $\phi$ is continuous.
More generally, this is the case as long as $\mu\leq A\mbox{Cap}^{p}$
for some $p>1,$ where $\mbox{Cap }$denotes the set functional defined
by the global Bedford-Taylor capacity on $X.$ Finally, to prove the
existence of a solution $\phi$ for any $T-$invariant pluripolar
probability measure $\mu$ one uses, just as in \cite{bbgz}, a decomposition
argument to reduce the problem (using the continuity properties for
$MA_{g}$ established in section above) to the case when $\mu$ has
finite energy, or more precisely to the special case when $\mu$ is
in the weakly compact subset $\mathcal{A}$ of all probability measures
satisfying $\mu\leq A\mbox{Cap}^{p}$ for some fixed positive constants
$A$ and $p$ such that $p>1.$

\section{\label{sec:K=0000E4hler-Ricci-solitons}Kähler-Ricci solitons}

\subsection{\label{sub:Setup krs}Setup}

Our basic complex geometric data in this section will be a pair $(X,V)$
consisting of a Fano variety $X$ (i.e. $X$ is a projective normal
variety with log terminal singularities with the property that the
$\Q-$line bundle $-K_{X}$ is ample) and $V$ a complex holomorphic
vector field on the regular locus $X_{reg}$ (see \cite{bbegz} and
references therein for further background on Fano varieties). We will
denote by $\mbox{Aut}(X,V)$ the group of all biholomorphisms of $X$
which commute with the flow of $V$ (or equivalently: all automorphisms
$F$ such that $F_{*}V=V)$ and by $\mbox{Aut}(X,V)_{0}$ the connected
component of the identity.

Any given metric $\phi\in PSH(X,-K_{X})$ induces a measure $\mu_{\phi}$
on $X,$ which may be defined as follows: if $U$ is a coordinate
chart in $X_{reg}$ with local holomorphic coordinates $z_{1},...,z_{n}$
we let $\phi_{U}$ be the representation of $\phi$ with respect to
the local trivialization of $-K_{X}$ which is dual to $dz_{1}\wedge\cdots\wedge dz_{n}.$
Then we define the restriction of $\mu_{\phi}$ to $U$ as 
\begin{equation}
\mu_{\phi}=e^{-\phi_{U}}i^{n}dz_{1}\wedge d\bar{z}_{1}\cdots\wedge dz_{n}\wedge d\bar{z}_{n}\label{eq:def of mu phi}
\end{equation}
This expression is readily verified to be independent of the local
coordinates $z$ and hence defines a measure $\mu_{\phi}$ on $X_{reg}$
which we then extend by zero to all of $X.$ The Fano variety $X$
has\emph{ log terminal (klt) singularities} precisely when the total
mass of $\mu_{\phi}$ is finite for some $\phi\in PSH(X,L)$ (or equivalently:
the finiteness holds for any locally bounded metric $\phi).$ More
precisely, if $X$ has log terminal singularities then the the pull-back
of $\mu_{\phi}$ to any given resolution $X'$ of $X$ has has local
densities in $L_{loc}^{p}$ for some $p>1$ (see \cite{bbegz} for
the equivalence with the usual algebraic definition involving discrepancies
on $X').$

We will say that a Kähler metric $\omega$ defined on the regular
part $X_{reg}$ of a normal Fano variety $X$ is a \emph{(singular)
Kähler-Ricci soliton} on $X$ if the metric $\omega$ solves the equation
\begin{equation}
\mbox{Ric \ensuremath{\omega=\omega}}+L_{V}\omega,\label{eq:k-r soliton eq for metric omega}
\end{equation}
 for some complex holomorphic vector field $V$ on $X_{reg}$ (where
$L_{V}$ denotes the Lie derivative along $V$) and moreover the volume
of the metric $\omega$ on $X_{reg}$ is maximal in the sense that
it coincides with the global algebraic top intersection number $c_{1}(-K_{X})^{n}$
(then the pair $(\omega,V)$ will also be called a Kähler-Ricci soliton).
In the case when $X$ is smooth it is well-known that $(\omega,V)$
is a Kähler-Ricci soliton precisely when $\omega=\omega_{\phi}$ for
a unique smooth metric $\phi$ in $\mathcal{H}(-K_{X}),$ which is
invariant under $\mbox{Im}V,$ satisfying the equation 
\begin{equation}
MA_{g_{V}}(\phi)=\mu_{\phi}\label{eq:krs eq for phi}
\end{equation}
where $g_{V}$ is the normalized exponential function $g_{V}$ on
the Lie algebra of the real torus $T$ acting on $(X,-K_{X})$ generated
by $V,$ i.e. $g_{V}(m_{\phi})=e^{f_{\phi}}/C,$ where $f_{\phi}$
is the Hamiltonian function on $X$ determined by the canonical lift
of $V$ to $-K_{X}$ (compare section \ref{sub:Interlude-I:-the vector}).
Indeed, denoting by $\psi_{\omega}$ the metric on $-K_{X}$ determined
by the volume form of $\omega$ we have, by definition, that $\mbox{Ric \ensuremath{\omega_{\phi}}}$is
the curvature form of $\psi_{\omega}$ and since $L_{V}\omega_{\phi}=dd^{c}f_{\phi}$
the equation \ref{eq:k-r soliton eq for metric omega} holds iff $\psi_{\omega}-(\phi+h_{\omega})$
is constant, i.e. iff the equation \ref{eq:krs eq for phi} holds. 

As a final matter of notation we will say that a (singular) metric
$\phi$ in $PSH(-K_{X})^{T}$ is a \emph{weak Kähler-Ricci soliton
}(wrt $V$ generating an action of $T)$ if it has full Monge-Ampère
mass and the equation \ref{eq:krs eq for phi} holds on $X.$

\subsection{Tian-Zhu's modified functionals and Futaki invariant }

Recall that in the smooth case the Mabuchi K-energy functional $\mathcal{M}$
of a Fano manifold $X$ equipped with its standard polarization $L=-K_{X}$
is, up to an additive constant, defined by the property that its differential
on $\mathcal{H}(X,-K_{X})$ is given by

\[
d\mathcal{M}_{|\phi}=-(\mbox{Ric \ensuremath{\omega_{\phi}}}-\omega_{\phi})\wedge\omega_{\phi}^{n}
\]
In particular its critical points are Kähler-Einstein metrics on $X.$
Similarly, in the presence of a vector field $V$ the modified K-energy
functional of Tian-Zhu \cite{t-z1b}, that we shall denote by $\mathcal{M}_{V},$
is obtained by replacing $\mbox{Ric}\ensuremath{\omega_{\phi}}$ with
the modified Ricci curvature $\mbox{Ric}\ensuremath{\omega_{\phi}}-L_{V}\omega_{\phi}$
(compare \cite{t-z1b}). In order to deal with singular metrics varieties
we will modify the singular setup for Kähler-Einstein metrics introduced
in \cite{bbegz}, by defining the\emph{ modified Mabuchi functional}
$\mathcal{M}_{V}(\phi)$ on the space $\mathcal{E}(-K_{X})^{T}$ by
the formula 
\[
\mathcal{M}_{V}(\phi):=F_{V}(MA(\phi)),
\]
 where $F_{V}$ is the \emph{modified free energy functional} on the
space of all $T-$invariant probability measure with finite energy,
defined by 
\[
F_{V}(\mu):=-E_{V}(\mu)+H\ensuremath{(\mu,\mu_{\phi_{0}}),}
\]
where $E_{V}(\mu):=E_{g_{V}}(\mu)$ is the $g-$energy defined in
the beginning of section \ref{sub:Monge-Amp=0000E8re-equations-and}
and $H(\mu,\mu_{\phi_{0}})$ is the \emph{entropy of $\mu$ relative
to $\mu_{\phi_{0}},$} i.e. 
\[
H(\mu,\mu')=\int_{X}\log(d\mu/d\mu')d\mu
\]
 if $\mu$ is absolutely continuous wrt $\mu'$ and $H(\mu,\mu')=\infty$
otherwise. Next, we define the \emph{modified Ding functional }$\mathcal{D}_{V}$
by 

\[
\mathcal{D}_{V}(\phi)=-\mathcal{E}_{V}(\phi)+\mathcal{L}(\phi),\,\,\,\,\mathcal{L}(\phi)=-\log\int_{X}e^{-\phi}
\]
(compare \cite{z,t-z1} for the smooth case). We will have great use
for the convexity properties of the functional $\mathcal{L}(\phi),$
originating in the work of Berndtsson \cite{bern1}, which combined
with Prop \ref{prop:energy along geodesics} gives the following
\begin{prop}
\label{prop:conv of mod ding}Let $\phi_{t}$ be a bounded geodesic.
Then the function $t\mapsto\mathcal{L}(\phi_{t})$ is convex on $]0,1[.$
Moreover, if the function is affine, then there exists a family of
automorphisms $F_{t}$ in $\mbox{Aut}(X)_{0}$ such that $F_{t}^{*}\phi_{t}=\phi_{t}.$
As a consequence, by Prop \ref{prop:energy along geodesics}, the
function $t\mapsto\mathcal{D}_{V}(\phi_{t})$ is also convex with
the same necessary conditions to be affine.
\end{prop}
Next, using the thermodynamic formalism in \cite{be2} we also have
the following
\begin{prop}
\label{prop:relation between mod ding and mab}The functional $\mathcal{D}_{V}$
is bounded iff the functional $\mathcal{M}_{V}$ is bounded and \emph{their
infima coincide.} Moreover, in general, $\mathcal{M}_{V}(\phi)\geq\mathcal{D}_{V}(\phi)$
with equality iff $\phi$ is a Kähler-Ricci soliton.\end{prop}
\begin{proof}
This is shown precisely as in \cite{be2} using that, by definition,
$\mathcal{D}_{V}=-E_{V}^{*}+H^{*}$ where the upper star denotes the
Legendre transform between convex functionals on the space $\mathcal{P}$
of all probability measures on $X$ and concave functionals on the
the space of all continuous metrics on $L$ (using the pairing $\left\langle \phi,\mu\right\rangle :=-\int(\phi-\phi_{0})d\mu$
between the latter spaces).\end{proof}
\begin{thm}
\label{thm:variational prop of mod mab ding}Let $X$ be a Fano variety
with log terminal singularities. Then the following is equivalent
for a metric $\phi$ in \emph{$\mathcal{E}^{1}(-K_{X}):$} 
\begin{itemize}
\item $\phi$ minimizes the modified Ding functional $\mathcal{D}_{V}$
\item $\phi$ minimizes the modified Mabuchi functional $\mathcal{M}_{V}$ 
\item $\phi$ is a weak Kähler-Ricci soliton for $(X,V).$ 
\end{itemize}
\end{thm}
\begin{proof}
If $\phi$ is a weak Kähler-Ricci soliton for $(X,V),$ then, by the
convexity in the previous theorem it minimizes $\mathcal{D}_{V}.$
Conversely, if $\phi\in\mathcal{E}^{1}(-K_{X})$ minimizes $\mathcal{D}_{V},$
then we deduce that it satisfies the Kähler-Ricci soliton equation,
by repeating the projection argument used in the proof of the existence
part in Theorem \ref{thm:existence uni ma eq energy} (replacing the
functional $\mathcal{L}_{\mu}$ used there with the functional $\mathcal{L}).$
The proof is now concluded by invoking the previous proposition.
\end{proof}

\subsubsection{The modified Futaki invariant}

Recall that, in the smooth case, the modified Futaki invariant of
Tian-Zhu, attached to $(X,V)$ \cite{t-z1b}, may be defined as the
following real valued function on the space of all holomorphic vector
fields $W$ generating a $\C^{*}-$action 
\begin{equation}
\mbox{Fut}_{V}(W):=\frac{\mathcal{M}(\phi_{t}^{W})}{dt},\,\,\,\,\phi_{t}^{W}:=\mbox{exp \ensuremath{(tW)^{*}\phi,}}\label{eq:def of fut inv as deriv along mab}
\end{equation}
where $\phi$ is any metric in $\mathcal{H}(X,-K_{X})$ which is invariant
under the $S^{1}-$action generated by the imaginary part of $W.$
Strictly speaking the original definition in\textbf{ }\cite{t-z1b}
may, in our notation, be written as 
\[
\mbox{Fut}_{V}(W)=c_{n}\int W(f_{\phi_{0}}^{W}-h_{\phi_{0}})\omega_{\phi_{0}}^{n},
\]
 where $f_{\phi_{0}}^{W}$ is the Hamiltonian function corresponding
to $W$ (compare the Lemma below) and $h_{\phi_{0}}$ is the normalized
Ricci potential of $\omega_{\phi_{0}}.$ The equivalence with the
definition \ref{eq:def of fut inv as deriv along mab} follows from
standard integration by parts. Anyway, here it will be convenient
to use the definition \ref{eq:def of fut inv as deriv along mab}
which applies verbatim in the general singular setting and which is
in line with the notion of modified K-stabilitiy introduced below.
The modified Futaki invariant is independent of $\phi$ and always
finite, as follows, for example, from the following alternative formula:
\begin{lem}
\label{lem:alt def of fut}Let $X$ be a Fano variety with log terminal
singularities. Then 
\[
\mbox{Fut}_{V}(W)=-\frac{d\mathcal{E}_{V}(\phi_{t}^{W})}{dt}=-\int_{X}f_{\phi_{0}}^{W}\exp(f_{\phi_{0}}^{V})MA(\phi_{0})
\]
 which is independent of $\phi_{0},$ where $f_{\phi_{0}}^{W}$ and
$f_{\phi_{0}}^{V}$ denote the Hamiltonian functions determined by
the canonical lift to $-K_{X}$ of the vector fields $V$ and $W,$
respectively.\end{lem}
\begin{proof}
By definition, $\mathcal{M}_{V}(\phi)-\mathcal{D}_{V}(\phi)=C\int h_{\phi}e^{h_{\phi}}\omega_{\phi}^{n}$
where $C$ is a constant and $h_{\phi}$ is the normalized Ricci potential
of $\phi,$ i.e. $h_{\phi}=\log MA(\phi)-(\phi-\mathcal{L}(\phi)).$
But the latter integral is invariant under $\phi\mapsto F^{*}(\phi)$
for any automorphism $F$ of $X$ and in particular for $F=\mbox{exp}\ensuremath{(tW)^{*}\phi}.$
In other words, $\mathcal{M}_{V}(\phi_{t}^{W})-\mathcal{D}_{V}(\phi_{t}^{W})$
is independent of $t$ and in particular the derivatives wrt $t$
vanish, which proves the first identity in the Lemma. The second identity
follows directly from the very definition of the objects involved
and the independence wrt $\phi_{0}$ then follows from the independence
wrt $\phi_{0}$ of the D-H measure attached to the torus $S^{1}\times T$
generated by $(W,V).$
\end{proof}
In section \ref{sub:The-quantized-K=0000E4hler-Ricci} a  ``quantized''
version of the previous lemma will be given which provides an algebraic
(i.e. metric free) formulation of the modified Futaki invariant. 
\begin{prop}
\label{prop:mab bounded from below implies fut 0}Assume that $(X,V)$
is such that $\mathcal{M}_{V}$ (or equivalently $\mathcal{D}_{V}$)
is bounded from below. Then $\mbox{Fut}_{V}(W)=0$ for any vector
field $W.$ In particular, this is the case if $(X,V)$ admits a Kähler-Ricci
soliton.\end{prop}
\begin{proof}
Assuming that $\mathcal{D}_{V}(\phi_{t}^{W})$ is bounded from below
and using that $\mathcal{L}(\phi_{t}^{W})$ is constant wrt $t$ it
follows that $\mathcal{E}_{V}(\phi_{t}^{W})$ is bounded from below,
forcing $-\frac{d\mathcal{E}_{V}(\phi_{t}^{W})}{dt}\geq0.$ Finally,
replacing $W$ with $-W$ forces the converse inequality $\frac{d\mathcal{E}_{V}(\phi_{t}^{W})}{dt}\geq0.$ 
\end{proof}

\subsection{Regularity: Proof of Theorem \ref{thm:krs regularity intro}}

\emph{Global continuity}: Given a smooth Kähler-Ricci soliton $\omega$
it extends, by normality to a unique positive current $\omega$ in
$c_{1}(-K_{X}).$ In particular, $\omega=\omega_{\phi}$ for a metric
$\phi$ in $PSH(X,-K_{X}).$ Set $\psi:=-\log MA_{g}(\phi)$ which
defines a smooth metric on $-K_{X}\rightarrow X_{reg}$ whose curvature
form $\omega_{\psi}$ coincides with $\omega_{\phi}$ on $X_{reg}$
(by the Kähler-Ricci soliton equation). In particular, $\psi$ extends
to a unique element in $PSH(X,-K_{X}),$ whose curvature current will
still be denoted by $\omega_{\psi}.$ But then $\omega_{\phi}=\omega_{\psi}$
on $X_{reg}$ implies, by normality, that $\phi=\psi+C$ on $X$ for
a constant $C$ and hence restricting to $X_{reg}$ reveals that the
equation \ref{eq:krs eq for phi} for $\phi$ holds on $X_{reg},$
up to replacing $\phi$ by $\phi-C.$ Moreover, since $MA_{g}(\phi)$
does not charge pluripolar sets the equation \ref{eq:krs eq for phi}
holds globally on $X.$ In particular, the mass of the measure $\mu_{\phi}$
with density $e^{-\phi}$ is bounded from above by $\int_{X}MA_{g}(\phi)$
for $\phi\in PSH(X,-K_{X})$ and hence finite, which shows that $X$
has log terminal singularities. Moreover, by the volume assumption,
$\int_{X}MA(\phi):=\int_{X_{reg}}\omega_{\phi}^{n}/c^{1}(-K_{X})=1$
and hence the singular metric $\phi$ in $PSH(X,-K_{X})$ has maximal
Monge-Ampère mass. But then $\phi$ has no Lelong numbers and the
measure corresponding to $\mu_{\phi}$ on any smooth resolution $X'$
of $X$ has a density in $L_{loc}^{p}$ for any $p>1$ (see the appendix
in \cite{bbegz}). Finally, combining Cor \ref{cor:comp of g ma and ma}
with the global equality $MA_{g}(\phi)=\mu_{\phi}$ gives $MA(\phi)\leq C\mu_{\phi}$
and hence, by \cite{e-g-z}, the $L^{p}-$property of $\mu_{\phi}$
implies that $\phi$ is continuous. 

\emph{Smoothness:} To conclude the proof we need to show that any
continuous metric $\phi$ satisfying the equation \ref{eq:krs eq for phi}
is in fact smooth on the regular locus of $X.$ To this end we let
$\phi_{t}$ be the Kähler-Ricci flow starting at the continuous metric
$\phi$(constructed by Song-Tian \cite{so-t}; compare the proof of
Theorem \ref{thm:conv of krf text} below). By \cite{so-t} $\phi_{t}$
is smooth on $X_{reg}$ for any $t>0.$ Next, we note that $\psi_{t}:=\exp(tV)^{*}\phi_{t}$
evolves according to the modified Kähler-Ricci flow and in particular
$\mathcal{D}_{V}(\psi_{t})$ is decreasing in $t$ (as in the proof
of Theorem \ref{thm:conv of krf text}). Since $\psi_{0}$ minimizes
$\mathcal{D}_{V}$ (by Theorem \ref{thm:variational prop of mod mab ding}),
it then follows that so does $\psi_{t}$ for any $t>0$ and hence
invoking Theorem \ref{thm:variational prop of mod mab ding} again
gives that $\psi_{t}$ is a Kähler-Ricci soliton for $(X,V)$ and
hence a stationary point for the modified Kähler-Ricci flow, i.e.
$\psi_{t}$ is independent of $t>0.$ But then letting $t\rightarrow0$
we conclude that $\psi_{t}=\psi_{0}=\phi$ for any $t>0,$ which shows
that $\phi$ is smooth on $X_{reg},$ as desired.

\subsection{Uniqueness and reductivity}
\begin{thm}
\label{thm:uniqueness text}Let $X$ be a Fano variety and $(\omega_{0},V_{0})$
and $(\omega_{1},V_{1})$ be two Kähler-Ricci solutions. Then there
exists $F\in\mbox{Aut}(X)_{0}$ such that $F^{*}\omega_{1}=\omega_{0}$
and $F^{*}V_{1}=V_{0}.$ In particular, if $V_{0}=V_{1}$ then $F\in\mbox{Aut}(X,V_{0})_{0}.$
More precisely, in the general case $F$ can be taken as the time
one flow map defined by a real vector field $Y$ on $X$ such that
$JY$ is in the compact isometry group $Iso(X,\omega_{1}).$ \end{thm}
\begin{proof}
Let us first fix a vector field $V$ and let $\phi_{0}$ and $\phi_{1}$
be two weak Kähler-Ricci solitons, defined with respect to $V$ and
in particular $T-$invariant. By Theorem \ref{thm:krs regularity intro}
$\phi_{0}$ and $\phi_{1}$ are locally bounded. Denote by $\phi_{t}$
the corresponding bounded geodesic curve in $PSH(-K_{X})^{T}.$ Consider
the function $f(t):=\mathcal{D}_{V}(\phi_{t}).$ By the minimization
property in Theorem \ref{thm:variational prop of mod mab ding} and
the fact that $f(t)$ is convex (by Prop \ref{prop:conv of mod ding})
it follows that $f(t)$ is affine on $[0,1].$ In particular, by Prop
\ref{prop:conv of mod ding}, $\mathcal{L}(\phi_{t})$ is affine.
But then Prop \ref{prop:conv of mod ding} implies that there exists
$F_{t}$ in $\mbox{Aut}(X)_{0}$ such that $\phi_{t}=F_{t}^{*}\phi_{0}.$
Next, we note that $F_{t}$ preserves $V,$ i.e. $F_{t}$ is in $\mbox{Aut}(X,V)_{0}.$
Indeed, $\phi_{t}$ minimizes $\mathcal{D}_{V}$ for any $t$ and
hence by Theorem \ref{thm:variational prop of mod mab ding} $(\phi_{t},V)$
is a weak Kähler-Ricci soliton. But since $\phi_{t}=F_{t}^{*}\phi_{0}$
it then follows that $L_{Z}\omega_{0}=0$ for $Z:=F_{t}^{*}V-V.$
But then $dd^{c}f=0$ for $f$ the Hamiltonian function of $F_{t}^{*}V-V$
defined on $X_{reg},$ which, by normality, forces $f=0,$ i.e. $F_{t}^{*}V=V,$
as desired. To see that there exists a vector field $W$ as in the
theorem, we recall that the proof of the convexity properties of $\mathcal{D}_{V}$
in \cite{bern1,bbegz} realizes $F_{t}$ by integrating a family of
holomorphic vector fields $V_{t}$ where $V_{t}$ has Hamiltonian
function $h_{t}:=d\phi_{t}/dt$ and hence $W_{t}:=\mbox{Re}V_{t}=J\mbox{Im}V_{t},$
where $\mbox{Im}V_{t}$ preserves $\omega_{t}.$ Moreover, as pointed
out in the end of the exposition of the proof in the appendix of \cite{c-d-s}
(III) it follows from the relation $\phi_{t}=F_{t}^{*}\phi_{0}$ combined
with the fact that $\phi_{t}$ is a smooth geodesic on $X_{reg}$
that $F_{t}$ is, in fact, a one-parameter group generated by the
vector field $Y:=Y_{0}.$ 

Finally, the case of different vector fields $V_{0}$ and $V_{1}$
can be reduced to the previous case, by the arguments in \cite{t-z1b}.
For completeness we recall the elegant argument. First, by Iwasawa's
theorem, we may, after perhaps replacing $V_{1}$ with $F_{*}V_{1},$
for some $F\in\mbox{Aut}(X)_{0},$ assume that the one-parameter isometry
groups generated by the imaginary parts of $V_{0}$ and $V_{1},$
respectively, are contained in the same compact subgroup $K$ of isometries
as $(X,\omega_{0}).$ In fact, since $[V,W]=0$ (by the argument in
the beginning of the proof of the next corollary) we may even assume
that they are contained in Lie algebra of the same subtorus $T'$
of $K.$ Now consider the functional on the Lie algebra of $T',$
defined by $\mathcal{F}(V):=\int_{X}\exp(f_{\phi_{0}}^{V})\omega_{\phi_{o}}^{n}$
(which is independent of $\phi_{0}).$ Its differential is given by
$(d\mathcal{F}_{|V})(W):=\int_{X}f_{\phi_{0}}^{W}\exp(f_{\phi_{0}}^{V})\omega_{\phi_{o}}^{n}$
and hence if $V$ is defined by a Kähler-Ricci soliton, then it is
a critical point of $V$ (compare Prop \ref{prop:mab bounded from below implies fut 0}).
But the functional $\mathcal{F}(V)$ is strictly convex and hence
$V$ is uniquely determined, as desired. \end{proof}
\begin{cor}
\label{cor:reduct}Let $X$ be a Fano variety and $V$ a holomorphic
vector field on $X.$ If $(X,V)$ admits a Kähler-Ricci soliton $\omega,$
then $\mbox{Aut}(X,V)_{0}$ is reductive. More precisely, the group
$\mbox{Aut}(X,V)_{0}$ may be identified with the complexification
of the compact group $Iso(X,\omega)$ consisting of all isometries
of $(X,\omega),$ i.e. all elements $F$ in $\mbox{Aut}(X)_{0}$ such
that $F^{*}\omega=\omega.$\end{cor}
\begin{proof}
First observe that $Iso(X,\omega)\subset\mbox{Aut}(X,V)_{0}.$ Indeed,
by assumption, $F_{*}\omega=\omega$ and hence applying $F^{*}$ to
both sides in the Kähler-Ricci soliton equation forces $L_{F_{*}V-V}\omega=0.$
But then it follows as in the proof of the previous theorem that $F_{*}V-V=0.$
Hence, $F$ is in $\mbox{Aut}(X,V)_{0},$ as desired. Moreover, the
group $K:=Iso(X,\omega)$ is a compact Lie group (compare the proof
of Lemma \ref{lem:vector field torus}) and we denote by $K_{c}$
its complexification. Given the previous theorem  the rest of the
argument proceeds exactly as in the elegant argument for the Kähler-Einstein
case in \cite{c-d-s}: let $g$ be an element in the Lie group $G:=\mbox{Aut}(X,V)_{0}.$
Then $g^{*}\omega$ is also a Kähler-Ricci soliton wrt $V$ and hence,
by the previous theorem, there exists a one parameter group $F_{t}$
generated by a vector field $Y$ in the Lie algebra of $K_{c}$ and
$F_{1}^{*}g^{*}\omega=\omega.$ But then $g\circ(F_{1})^{-1}$ is
in $K$ and since $F_{1}$ is in $K_{c}$ we conclude (since $K\circ K_{c}\subset K_{c}$)
that so is $g.$ 
\end{proof}

\subsection{K-stability}

Recall that a special test configuration for a Fano variety $X$ is
defined by a variety $\mathcal{X}$ equipped with a $\C^{*}-$action
$\rho$ and an equivariant morphism $\pi$ to $\C$ (with its standard
$\C^{*}-$action) such that the (scheme theoretic) fibers are Fano
varieties with log terminal singularities and $\pi^{-1}\{1\}=X.$
We will denote by $\mathcal{W}$ the holomorphic vector field on $\mathcal{X}$
generating the $\C^{*}$ action, which restricts to a holomorphic
vector field $W_{0}$ on the central fiber $X_{0}.$ More generally,
we will say that $(\mathcal{X},\rho_{\mathcal{W}},\mathcal{V})$ is
a\emph{ special test configuration for $(X,V),$ }if the vector field
$V$ on the generic fiber $X$ is the restriction of a holomorphic
vector field $\mathcal{V}$ on $\mathcal{X}$ preserving the fibers
of $\mathcal{X}$ and commuting with $\mathcal{W}.$ 
\begin{defn}
The modified Futaki invariant $\mbox{Fut}(\mathcal{X},\rho_{\mathcal{W}},\mathcal{V})$
of a test configuration $(\mathcal{X},\rho_{\mathcal{W}},\mathcal{V})$
is defined as the modified Futaki invariant $F_{V_{0}}(W_{0})$ of
the induced holomorphic vector field $W_{0}$ on the central fiber
$X_{0}.$ 
\end{defn}
More concretely, realizing $X$ as subvariety of a projective space
$\P^{N}$ in such a way that a multiple of $-K_{X}$ gets identified
with $\mathcal{O}(1)_{|X}$ it is enough to consider test configurations
$(\mathcal{X},\rho,\mathcal{V})$ such that $\mathcal{X}\subset\P^{N}\times\C$
with $\rho$ the restriction to $\mathcal{X}$ of a $\C^{*}-$action
on $\P^{N}$ preserving $\mathcal{X}$ and $\mathcal{V}$ the restriction
to $\mathcal{X}$ of a holomorphic vector field on $\P^{N}$ generating
(in the sense of Lemma \ref{lem:vector field torus}) a torus action
$T_{c}$ on $\P^{N}$ commuting with $\rho$ and preserving the fibers
of $\mathcal{X}$ and coinciding on the generic fiber $X$ with the
torus action generated by $V.$ 
\begin{defn}
A pair $(X,V)$ consisting of a Fano variety $X$ with log terminal
singularities and a holomorphic vector field $V$ on $X$ is said
to be K-polystable if the modified Futaki invariant of any special
test configuration $(\mathcal{X},\rho_{\mathcal{W}},\mathcal{V})$
for $(X,V)$ is non-negative and zero iff $(X_{0},V_{0})$ is isomorphic
to $(X,V).$
\end{defn}

\subsubsection{Proof of Theorem \ref{thm:krs implies k-stab intro}}

First recall that a test configuration $(\mathcal{X},\rho_{\mathcal{W}})$
for a Fano variety $X$ together with a initial continuous metric
$\phi_{0}$ in $PSH(-K_{X})$ induces a geodesic ray $\phi_{t}$ of
continuous metrics in $PSH(-K_{X}).$ The curve $\phi_{t}$ may, using
the $\C^{*}$ action $\rho_{\mathcal{W}},$ be identified with an
$S^{1}-$invariant continuous metric $\Phi$ in $PSH(M,-K_{M/\Delta})$
where $M:=\mathcal{X}_{|\Delta}$ and $\Delta$ is the unit-disc in
$\C$ (see \cite{be4} and references therein) and $\Phi$ satisfies
$(dd^{c}\Phi)^{n+1}=0$ in the interior of $M$ and may be identified
with $\phi_{0}$ on $\partial M.$ In the case when $(\mathcal{X},\rho_{\mathcal{W}},\mathcal{V})$
is a test configuration for $(X,V)$ and $\phi_{0}$ is invariant
under the real torus $T$ induced by $V$ it follows (from the uniqueness
of solutions to the Dirichlet problem above) that $\Phi$ is invariant
under the corresponding real torus action on $\mathcal{X}$ and hence
$\phi_{t}$ is a geodesic ray in $PSH(-K_{X})^{T}.$

Now, if $(X,V)$ admits a Kähler-Ricci soliton then we can take $\phi_{0}$
as the corresponding metric on $-K_{X}.$ Consider the functional
$f(t):=\mathcal{D}_{V}(t).$ By the convexity of $f(t)$ (see Prop
\ref{prop:conv of mod ding}) we then get 
\begin{equation}
0\leq\lim_{t\rightarrow\infty}\frac{df(t)}{dt}_{|t=0}=\lim_{t\rightarrow\infty}-\frac{d\mathcal{E}_{V}(\phi_{t})}{dt}_{|t=0}+0,\label{eq:ineq in proof of kstab}
\end{equation}
using in the last equality that 
\[
\lim_{t\rightarrow\infty}\frac{d\mathcal{L}(\phi_{t})}{dt}=0
\]
 for a special test configuration, as shown in \cite{be4} (we recall
that the key point is the vanishing of the Lelong number of the $L^{2}-$metric
on the direct image of $-K_{\mathcal{X}/\Delta}).$ Hence, by combining
the following proposition with Lemma \ref{cmm:alt def of fut} it
follows that $0\leq\mbox{Fut}(\mathcal{X},\rho_{\mathcal{W}},\mathcal{V}).$
\begin{prop}
Let $(\mathcal{X},\rho_{\mathcal{W}},\mathcal{V})$ be special test
configuration for $(X,V),$ $\phi_{0}$ a continuous metric in $PSH(-K_{X})^{T}$
and denote by $\phi_{t}$ the corresponding geodesic ray. Then 

\begin{equation}
\frac{d\mathcal{E}_{V}(\phi_{t})}{dt}=\frac{d\mathcal{E}_{V_{0}}(\phi_{t}^{W})}{dt}\label{eq:statement of prop derivatie of V-energy along geodesic}
\end{equation}
More generally, the equality holds when $t\rightarrow\infty$ if $\phi_{t}$
is replaced by the subgeodesic defined by a locally bounded metric
$\Phi$ on $-K_{\mathcal{X}/\Delta}.$ \end{prop}
\begin{proof}
We start by noting that the last statement about subgeodesics follows
from the first one about geodesics. Indeed, if $\phi_{t}$ and $\psi_{t}$
are such geodesics and sub-geodesics, respectively, then $|\phi_{t}-\psi_{t}|\leq C$
for a constant independent of $t$ and hence $f(t):=d\mathcal{E}_{V}(\psi_{t})-\mathcal{E}_{V}(\phi_{t})$
is a bounded convex function. In particular, its derivative tends
to zero as $t\rightarrow\infty.$ Next, we first assume that $X$
is smooth. By Prop \ref{prop:energy along geodesics} $\frac{d\mathcal{E}_{V}(\phi_{t})}{dt}$
is constant and equal to the total integral of $\frac{d\phi_{t}}{dt}MA_{g_{V}}(\phi_{t}).$
In other words, $\frac{d\mathcal{E}_{V}(\phi_{t})}{dt}$ is equal
to the first moment $\int_{\R}wd\gamma_{\phi_{t}}(w)$ of the probability
measure 
\[
\gamma_{\phi_{t}}:=(\frac{d\phi_{t}}{dt})_{*}MA_{g_{V}}(\phi_{t})
\]
 on $\R.$ The previous proposition can now be obtained as a special
case of the following convergence of probability measures on $\R$
(that holds for any $t\in[0,\infty[):$ 
\begin{equation}
\gamma_{\phi_{t}}=\gamma_{(V_{0},W_{0})}:=h_{*}^{W_{0}}(MA_{g_{V_{0}}}(\phi_{t}^{W_{0}}))=\lim_{k\rightarrow\infty}\frac{1}{N_{k}}\sum_{l=1}^{N_{k}}\exp(v_{l}^{(k)}/k)\delta_{w_{l}^{(k)}/k}\label{eq:gamma phi t as limit of}
\end{equation}
 where $(v_{l}^{(k)},w_{l}^{(k)})$ are the joint eigenvalues on $H^{0}(X_{0},-kK_{X_{0}})$
of the real parts of $V_{0}$ and $W_{0},$ respectively (compare
Prop \ref{prop:alg formula for fut}). This result is obtained by
adapting the proof of the main result in \cite{hi} (concerning the
case when there is no torus action, i.e. $g=1)$ to the present torus
setting. The point is that the commuting pair $(\mathcal{V},\mathcal{W})$
of vector fields on $\mathcal{X}$ generate an $m+1-$dimensional
torus $S^{1}\times T$ on $\mathcal{X}.$ In particular, integrating
formula \ref{eq:gamma phi t as limit of} against the linear test
function $w$ gives 
\[
\frac{d\mathcal{E}_{V}(\phi_{t})}{dt}=\int_{\R}wd\gamma_{\phi_{t}}(w)=\lim_{k\rightarrow\infty}\frac{1}{Nk_{k}}\sum_{l=1}^{N_{k}}\exp(v_{l}^{(k)}/k)w_{i}^{(k)}
\]
 and invoking Prop \ref{prop:alg formula for fut} thus concludes
the proof of formula \ref{eq:statement of prop derivatie of V-energy along geodesic}. 
\end{proof}
Next, in the case when $0=\mbox{Fut}(\mathcal{X},\rho_{\mathcal{W}},\mathcal{V})$
we must have equalities in the inequality in formula \ref{eq:ineq in proof of kstab}.
In particular, $\mathcal{D}(\phi_{t})$ is affine and hence, by Prop
\ref{prop:conv of mod ding}, there exists a family of holomorphic
vector fields $Y_{\tau}$ on $X$ inducing biholomorphisms $F_{\tau}$
from $X$ to $X_{\tau}$ such that $F_{\tau}^{*}\phi_{\tau}=\phi.$
Finally, using that $\Phi$ is continuous on all of $\mathcal{X}$
one obtains, just as in the \cite{be4} that $F_{\tau}$ converges
as $\tau\rightarrow0$ to a biholomorphism between $X$ and $X_{0}.$
The proof is concluded by noting that $F_{\tau}$ commutes with the
flow of $V.$ Indeed, as observed above $\phi_{\tau}$ minimizes $\mathcal{D}_{V}$
and is hence a Kähler-Ricci soliton wrt $V.$ But $F^{*}\phi_{\tau}$
is also a Kähler-Ricci soliton wrt $F_{\tau}^{*}V$ and hence (just
as in the in the proof of Theorem \ref{thm:uniqueness text}) it follows
that $F_{\tau}^{*}V=V,$ as desired.

\subsection{Analytic K-stability and the Kähler-Ricci flow}

Given a pair $(X,V)$ with $X$ a Fano variety (with log terminal
singularities) and $V$ a holomorphic vector field generating an action
of a torus $T,$ we will say that a functional $\mathcal{F}$ on $\mathcal{H}(-K_{X})^{T}$
is \emph{proper modulo $\mbox{Aut}(X,V)_{0}$ }if it is proper with
respect to the $\mbox{Aut}(X,V)-$invariant exhaustion function 
\[
\bar{J}(\phi):=\inf_{F\in\mbox{Aut}(X,V)}J(F^{*}\phi)
\]
 i.e. if $\mathcal{D}_{V}(\phi)\leq C$ implies that $\bar{J}(\phi)\leq C',$
where $C'$ only depends on $C.$ We will also say that a functional
$\mathcal{F}$ on $\mathcal{H}(-K_{X})^{T}$ is s\emph{trongly proper
(or coercive) modulo $\mbox{Aut}(X,V)_{0}$ }if there exist positive
constants $A$ and $B$ such that, for any $\phi\in\mathcal{H}(-K_{X})^{T}$
there exists $F\in\mbox{Aut}(X,V)_{0}$ satisfying 
\[
\mathcal{F}(\phi)\geq A\bar{J}-B
\]
We will say that $X$ is \emph{analytically K-polystable} if the modified
Mabuchi functional $\mathcal{M}_{V}$ on $\mathcal{H}(-K_{X})^{T}$
is proper modulo $\mbox{Aut}(X,V)_{0}$ and \emph{analytically strongly
K-polystable }if $\mathcal{M}_{V}$ is coercive modulo $\mbox{Aut}(X,V)_{0}$ 
\begin{thm}
\label{thm:proper ding mab mod aut implies existence}Let$(X,V)$
be a Fano variety with a holomorphic vector field. Assume that either
the modified Ding functional $\mathcal{D}_{V}$ or the modified Mabuchi
functional $\mathcal{M}_{V}$ is proper modulo $\mbox{Aut}(X,V)_{0}.$
Then $(X,V)$ admits a Kähler-Ricci soliton.\end{thm}
\begin{proof}
First assume that $\mathcal{D}_{V}$ is proper on $\mathcal{H}(-K_{X})^{T}$
modulo $\mbox{Aut}(X,V)_{0}$ and in particular bounded from below
(strictly speaking, by Prop \ref{prop:relation between mod ding and mab}
it is enough to consider the case when $\mathcal{M}_{V}$ is proper
on $\mathcal{H}(-K_{X})^{T}$ modulo $\mbox{Aut}(X,V)_{0},$ but for
future reference we start by considering the case of $\mathcal{D}_{V}$
separately). Then $\mathcal{D}_{V}$ has to be invariant under the
action of $\mbox{Aut}(X,V)_{0}.$ Indeed, $\mathcal{D}_{V}(\phi_{t}^{W})$
is bounded from below with respect to $t\in\R$ and $W$ a holomorphic
vector field commuting with $V,$ i.e. an element in the Lie algebra
of $\mbox{Aut}(X,V)_{0}$ and hence letting $t\rightarrow\pm\infty$
gives $d\mathcal{D}_{V}(\phi_{t}^{W})/dt=\lim_{t\rightarrow\infty}d\mathcal{D}_{V}(\phi_{t}^{W})/dt=0,$
as desired. By the assumed properness this means that we can take
a sequence $\phi_{j}$ in $\mathcal{H}(-K_{X})^{T}$ such that 
\[
\lim_{j\rightarrow\infty}\mathcal{D}_{V}(\phi_{j})=\inf_{\mathcal{H}(-K_{X})}\mathcal{D}_{V}=\inf_{\mathcal{E}(-K_{X})}\mathcal{D}_{V},\,\,\,\,\sup_{X}(\phi_{j}-\phi_{0})=0,\,\,\,-\mathcal{E}(\phi_{j})\leq C
\]
Hence, taking a sequence $\phi_{j}$ in $PSH(X,L)_{0}$ such that
$J_{\mu,g}(\phi_{j})\rightarrow\inf J_{\mu,g}$ the coercivity implies
that $-\mathcal{E}(\phi)(\phi_{j})\leq C$ for some constant $C.$
Accordingly, by the upper semicontinuity of $\mathcal{E}_{V}$ we
may assume, perhaps passing to a subsequence, that $\phi_{j}$ converges
in $L^{1}$ to a metric $\phi_{\infty}$ in $\mathcal{E}(-K_{X}).$
Moreover, as shown in \cite{bbgz} the functional $\mathcal{L}$ is
continuous on the closed subset $-\mathcal{E}(\phi)\leq C$ and hence
if follows that the limit $\phi_{\infty}$ realizes the infimum of
$\mathcal{D}_{V}.$ But then $(\phi_{\infty},V)$ is a Kähler-Ricci
soliton (by Theorem \ref{thm:variational prop of mod mab ding} ). 

Next, assuming instead that $\mathcal{M}_{V}$ is proper modulo $\mbox{Aut}(X,V)_{0}$
we deduce just as above that there exists a sequence $\phi_{j}$ converging
in $L^{1}$ to $\phi_{\infty}$ such that 
\[
\lim_{j\rightarrow\infty}\mathcal{M}_{V}(\phi_{j})=\inf_{\mathcal{H}(-K_{X})}\mathcal{M}_{V}=\inf_{\mathcal{E}(-K_{X})}\mathcal{M}_{V},\,\,\,\,\sup_{X}(\phi_{j}-\phi_{0})=0,\,\,\,-\mathcal{E}(\phi_{j})\leq C
\]
In particular, by the assumed properness, setting $\mu_{j}:=MA(\phi_{j})$
the entropies $H(\mu_{j})$ are uniformly bounded from above. But
then it follows from Theorem 2.17 in \cite{bbegz} that $\phi_{j}\rightarrow\phi$
in energy (i.e. in the so called strong topology introduced in \cite{bbegz})
and in particular $E(\mu_{j})\rightarrow E(MA(\phi_{\infty}).$ By
the lower semi-continuity of the entropy wrt the weak topology this
shows that $\phi_{*}$ minimizes $\mathcal{M}_{V}$ and hence, invoking
Theorem \ref{thm:variational prop of mod mab ding}) we conclude that
$(\phi_{\infty},V)$ is a Kähler-Ricci soliton.\end{proof}
\begin{rem}
\label{rem:k polystab vs analytic}Combining the previous theorem
with Theorem \ref{thm:krs implies k-stab intro} reveals that analytic
K-polystability of $(X,V)$ implies K-polystability of $(X,V).$ Is
seems natural to conjecture that the converse also holds, as well
as the equivalence between analytic K-polystability and strong analytic
K-polystability.\end{rem}
\begin{thm}
\label{thm:conv of krf text}Assume that $(X,V)$ is analytically
K-polystable. Then the Kähler-Ricci flow $\omega_{t}$ converges in
the weak topology of currents, modulo the action of the group $\mbox{Aut}(X,V)_{0},$
to a Kähler-Ricci soliton $\omega$ for $(X,V).$ \end{thm}
\begin{proof}
Let us denote by $\phi_{t}$ the Kähler-Ricci flow on the level of
metrics on $-K_{X},$ defined by 
\begin{equation}
\frac{d\phi_{t}}{dt}=\log\frac{MA(\phi_{t})}{\mu_{\phi_{t}}},\,\,\,\,\phi_{|t=0}=\phi_{0}\label{eq:krf for phi}
\end{equation}
where the initial data $\phi_{0}$ is assumed to be a continuous metric
in $PSH(X-K_{X})^{T}.$ In the singular case one demands that the
metric $\phi_{t}(x)$ be smooth on $]0,\infty[\times X_{reg}$ and
continuous on $X;$ the existence and uniqueness of such a flow was
established in \cite{so-t} (see also the exposition in \cite{b-g}).
Next, setting $\psi_{t}:=\exp(tV)^{*}\phi_{t}$ gives a solution to
the modified Kähler-Ricci flow $\psi_{t}$ obtained by replacing $MA$
with $MA_{g_{V}}.$ As is well-known, at least in the smooth case,
the modified Ding functional $\mathcal{D}_{V}$ is decreasing along
the modified Kähler-Ricci flow. In the singular case this is shown
by regularization, just as in the case $V=0$ considered in \cite{bbegz}
(using, in the case $V\neq0,$ the continuity properties of $\mathcal{E}_{V}$
in Prop \ref{prop:cont prop of g-energy}). In particular, by the
assumed properness mod $\mbox{Aut}(X,V)_{0}$ of $\mathcal{D}_{V},$
there exist $F_{t}$ in $\mbox{Aut}(X,V)_{0}$ such that $J(F_{t}(\psi_{t}))\leq C$
and $\mathcal{D}_{V}(F_{t}^{*}\psi_{t})=\mathcal{D}_{V}(\psi_{t}).$
The proof in the case $V=0$ in \cite{bbegz} can now be repeated,
mutis mutandis, to get that any subsequence $\psi_{t_{j}}$ converges
in energy to a minimizer $\psi_{*}$ of $\mathcal{D}_{V}$ and hence
(by Theorem \ref{thm:variational prop of mod mab ding}) to a Kähler-Ricci
soliton $\psi_{*}$ (wrt $V).$ However, a priori $\psi_{*}$ depends
on the subsequence. To get around this issue we observe that $F_{t}$
above may be taken so that $F_{t}\phi_{t}$ minimizes the functional
$J$ on the orbit $\mbox{Aut}(X,V)_{0}\phi_{t}$ and, in particular,
$J(F_{t}^{*}\phi_{t})\leq J(F^{*}\phi_{t})$ for any $F$ in $\mbox{Aut}(X,V)_{0}.$
But then, letting $t\rightarrow\infty$ reveals that any limit point
$\psi_{*}$ minimizes the functional $J$ on the space $\mathcal{H}_{KRS}$
of all KR-solitons on $(X,V),$ which by Theorem \ref{thm:uniqueness text}
may be identified with quotient $\mbox{Aut}(X,V)_{0}\psi_{*}/K,$
where $K$ is the stabilizer of $\psi_{0}.$ Finally, since $J$ is
strictly convex on $\mbox{Aut}(X,V)_{0}\psi_{*}/K$ equipped with
its natural Riemannian structure (where the geodesics are one parameter
subgroups) there is a unique such minimizer, which must hence correspond
to $\psi_{*}.$ Hence, the whole curve $\psi_{t}$ converges (up to
normalization) to $\psi_{*}$ in energy, which concludes the proof. 
\end{proof}

\section{\label{sec:The-quantized-setting}The quantized setting and semi-classical
asymptotics}

We continue with the setup in sections \ref{sub:The-torus-setting}
involving the holomorphic action of a real torus $T$ (and its complexification
$T_{c})$ on $(X,L),$ where the line bundle $L$ is assumed semi-positive
and big, i.e. it admits a smooth and semi-positively curved metric
$\phi_{0}$ whose curvature is strictly positive at some point. We
will fix such a $T-$invariant reference metric $\phi_{0}.$ Moreover,
we will denote by $V$ a fixed holomorphic vector field on $X$ generating
the action of $T$ in the sense of Lemma \ref{lem:vector field torus}
. In order to be consistent with the setup in Section \ref{sec:K=0000E4hler-Ricci-solitons}
we do allow $X$ to be singular (with log terminal singularities),
but in the proofs we may as well assume that $X$ is smooth by passing
to a resolution.

\subsection{Convergence of the spectral measure towards the Duistermaat-Heckman
measure}

Let $P_{k}:=\{\lambda_{i}^{(k)}\}\subset\Z^{m}$ be the set of all
weights for the action of the complex torus $T_{c}$ on the $N_{k}-$dimensional
vector space $H^{0}(X,kL)$ of all holomorphic sections on $X$ with
values in $kL,$ i.e. there is a decomposition

\[
H^{0}(X,kL)=\oplus_{\lambda_{i}^{(k)}\in P_{k}}E_{\lambda_{i}^{(k)}},\,\,\,\,\, s\in E_{\lambda_{i}^{(k)}}\iff\rho(\tau)^{*}s=\tau_{1}^{\lambda_{1}^{(k)}}\cdots\tau_{m}^{\lambda_{m}^{(k)}}s
\]
 Equivalently, we can view $\lambda_{i}^{(k)}$ as the joint eigenvalues,
counted with multiplicity (=the dimension of $E_{\lambda_{i}^{(k)}}),$
of the commuting infinitasimal actions of the real parts of the holomorphic
vector fields $V_{i}$ generating the action of $T_{c}.$ Here the
real part of a vector field $W$ on $X$ (with a fixed lift to $L)$
acts on $H^{0}(X,kL)$ by $\mbox{(Re}V)s:=\frac{d}{dt}_{|t=0}\exp(t\mbox{Re}V)^{*}s.$
We let 

\[
\nu_{k}:=\frac{1}{N_{k}}\sum_{i=1}^{N_{k}}\delta_{\lambda_{i}^{(k)}/k}
\]
be the corresponding normalized spectral measure on $\R^{m},$ supported
on the joint spectrum $P_{k}.$ Similarly, we denote by $\nu_{k}^{V}$
the spectral measure on $\R$ attached to the infinitesimal action
of the real part of $V$ on $H^{0}(X,kL):$ 
\[
\nu_{k}^{V}=\frac{1}{N_{k}}\sum_{i=1}^{N_{k}}\delta_{v_{i}^{(k)}/k}
\]
Hence, identifying $\mbox{Re}V$ with the corresponding element $\xi$
in $\R^{m},$ i.e. writing $\mbox{Re}V=\sum_{i=1}^{m}\xi_{i}\mbox{Re}V_{i},$
we have $v_{i}^{(k)}=\left\langle \lambda_{i}^{(k)},\xi\right\rangle .$ 
\begin{prop}
\label{prop:conv of spec meas} Assume that $L$ is semi-positive
and big. Then the spectral measures $\nu_{k}$ of the\textbf{ }infinitesimal
action on $H^{0}(X,kL)$ of the real torus $T$ converge weakly, as
$k\rightarrow\infty,$ to the D-H measure $\nu^{T}:=(m_{\phi})_{*}MA(\phi).$
In particular, if the torus $T$ is generated by a vector field $V$
corresponding to the element $\xi$ in $\R^{m},$ then the corresponding
spectral measures $\nu_{k}^{V}$ converge weakly to $\nu^{V}=f_{\phi*}^{V}MA(\phi),$
where $f_{\phi}^{V}$ is the Hamiltonian function determined by $V.$\end{prop}
\begin{proof}
For $L$ ample and $T=S^{1}$ this was shown in \cite{w} (but it
also follows from general results on torus actions in symplectic geometry).
The case of a torus clearly reduces, using that the generators commute,
to the case of $T=S^{1}.$ In order to deal with the case of $L$
semi-positive and big we fix a $T-$invariant metric $\phi$ which
is smooth and of non-negative curvature and note that, following the
argument in \cite{w}, the problem is reduced to establishing the
corresponding spectral asymptotics for a Toeplitz operator with smooth
symbol $f$ (in this case equal to $f_{\phi}^{V}).$ In the case of
$L$ semi-positive and big the asymptotics in question were obtained
in \cite{be1}. \end{proof}
\begin{cor}
\label{cor:number of eigenv}For $\lambda$ an interior point in $P$
denote by $N_{k}(\lambda)$ the number of eigensections in the $N_{k}-$
dimenstional space $H^{0}(X,kL)$ with joint eigenvalues $\lambda^{(k)}$
such that $\lambda^{(k)}/k\geq\lambda$ as vectors in $\R^{m}.$ Then,
for almost any $\lambda,$ 
\[
\lim_{k\rightarrow\infty}\frac{N_{k}(\lambda)}{N_{k}}=\int_{X}MA(P_{\lambda}\phi)
\]
where $\phi$ is any locally bounded $T-$invariant metric on $L.$\end{cor}
\begin{proof}
For $\phi$ smooth and of non-negative curvature this follows from
combing the previous proposition with formula \ref{eq:relation for ma of chi in lemma}
using $\chi_{\lambda}$ as a test function. But, by Prop \ref{prop:cont of ma with sing},
the mass of $MA(P_{\lambda}\phi)$ is the same if $\phi$ is replaced
with a locally bounded metric (since the action of the operator $P_{\lambda}$
only changes with a bounded term). \end{proof}
\begin{rem}
Conversely, if one takes the latter corollary as granted then it implies
the previous proposition by a standard measure theory argument (compare
the approach used in \cite{hi}).
\end{rem}

\subsection{The quantized energy functionals and their asymptotics}

Recall that for any positive integer $k$ the \emph{quantization at
level $k$ }of the space $\mathcal{H}(L)$ is the space $\mathcal{H}_{k}$
of all Hermitian metrics $H$ on the $N_{k}-$dimensional complex
vector space $H^{0}(X,kL)$ (see \cite{do3}). Fixing a reference
element $H_{0}$ the space $\mathcal{H}_{k}$ may be identified with
the symmetric space $GL(N_{k},\C)/U(N_{k})).$ We will equip $\mathcal{H}_{k}$
with the corresponding symmetric Riemannian metric and denote by $\mathcal{H}_{k}^{T}$
the corresponding $T-$invariant subspace. This means that the geodesics
in $\mathcal{H}_{k}$ corresponds to one-parameter subgroups in $GL(N_{k},\C).$
Recall also that there is a map, the ``Fubini-Study map'' 
\[
FS_{k}:\,\,\mathcal{H}_{k}\rightarrow\mathcal{H}(L),\,\,\, FS_{k}(H):=\sup_{s\in H_{0}(X,kL)}\log\left(\frac{|s|^{2}}{H(s,s)}\right)
\]
which is compatible with the torus action. Next, recall that $E_{\lambda^{(k)}}$
denotes the joint eigenspace in $H^{0}(X,kL)$ corresponding to the
joint eigenvalue $\lambda^{(k)}$ in $\R^{m}$ attached to the torus
$T.$ To any triple $(\phi,\mu,g)$ consisting of a $T-$invariant
smooth metric $\phi$ on $L$ with semi-positive curvature, a $T-$invariant
probability measure $\mu$ on $X$ (of finite energy) and a bounded
function $g$ on $P,$ we attach a $T-$invariant Hilbert norm (metric)
denoted by $\mbox{Hilb}(\phi,\mu,g),$ defined by 
\[
\mbox{Hilb}(\phi,\mu,g)(s_{i},s_{i})=g(\lambda_{i}^{(k)}/k)^{-1}\int_{X}|s_{i}|^{2}e^{-k\phi}d\mu\,\,\,\,\mbox{for \ensuremath{s_{i}\in E_{\lambda^{(k)}}\subset H^{0}(X,kL)}}
\]
and declaring that the different subspace $E_{\lambda^{(k)}}$mae
mutually orthogonal wrt $\mbox{Hilb}(\phi,\mu,g).$ For $g$ strictly
positive this defines a Hilbert norm on $H^{0}(X,kL),$ i.e. an element
in $\mathcal{H}_{k}^{T}$ and in general it defines a Hilbert norm
on the subspace of $H^{0}(X,kL)$ spanned by all eigensections such
that $g(\lambda_{i}^{(k)})>0.$

\subsubsection{The $g-$Bergman measure}

To a triple $\mbox{Hilb}(\phi,\mu,g)$ we associate the following
Bergman type function at level $k:$ 
\[
B_{(k\phi,\mu,g)}:=\sup_{s\in H_{0}(X,kL)}\frac{|s|_{k\phi}^{2}}{\mbox{Hilb}(k\phi,\mu,g)(s,s)}
\]
coinciding with the classical Bergman function for $g=1,$ also called
the density of states (see \cite{ze}). In the case when $\mu=dV$
for a fixed volume form $dV$ on $X$ we will drop the explicit dependence
of $\mu$ from the notation and simply write $B_{k\phi,g}:=B_{(k\phi,\mu,g)}$
which can thus be written as 

\[
B_{k\phi,g}:=\sum_{\lambda_{i}^{(k)}\in P_{k}}g(\lambda_{i}^{(k)}/k)B_{k\phi,\lambda^{(k)}}
\]
where $B_{k\phi,\lambda^{(k)}}$ is the ordinary Bergman function
of the subspace $E_{\lambda^{(k)}}$ of $H^{0}(X,kL).$ According
to the following proposition one can view the\emph{ $g-$Bergman measure}
$B_{k\phi,g}dV/N_{k}$ as the quantization, at level $k,$ of the
$g-$Monge-Ampère measure: 
\begin{prop}
\label{prop:conv of g-bergman}Assume that $\phi$ is in $\mathcal{H}(L)^{T}.$
Then the following convergence holds in the weak topology of measures
on $X$ 
\[
\frac{B_{k\phi,g}dV}{N_{k}}\rightarrow MA_{g}(\phi)
\]

Moreover, if $L$ is ample, $g$ is smooth and $\phi$ has positive
curvature, then the convergence above holds in the uniform topology
(on the level of densities). \end{prop}
\begin{proof}
\emph{Weak convergence:} Proceeding as in the proof of Theorem \ref{thm:g-ma is well-def and cont}
it is enough to prove the case when $g=\chi_{\lambda},$ so that $B_{k\phi,\chi_{\lambda}}$
is the Bergman function of the subspace generated by all eigensections
with joint eigenvalues $\lambda^{(k)}$ such that $\lambda^{(k)}/k\geq\lambda.$
First, by the local holomorphic Morse inequalities in \cite{be1}
the following point-wise upper bound holds: 
\[
\limsup_{k\rightarrow\infty}\frac{B_{k\phi,\chi_{\lambda}}dV}{N_{k}}\leq1_{\{dd^{c}\phi\geq0\}}MA(\phi)
\]
together with the uniform bound $\frac{B_{k\phi,\chi_{\lambda}}}{N_{k}}\leq C$.
Moreover, there exists a constant $C$ such that 
\begin{equation}
\frac{B_{k\phi,\chi_{\lambda}}dV}{N_{k}}\leq C\exp(-k(\phi-P_{\lambda}\phi)).\label{eq:exponential localiz}
\end{equation}
Indeed, by the uniform bound above $\phi_{k}:=\phi+\frac{1}{k}\log B_{k\phi,\chi_{\lambda}}-\log N_{k}-\log C\leq\phi.$
But then it follows from the definition of the envelope $P_{\lambda}\phi$
that $\phi_{k}\leq P_{\lambda}\phi$ (compare the proof of Lemma \ref{lem:Plambdaphi is local bounded on}),
which proves the inequality \ref{eq:exponential localiz}. All in
all this means that 
\[
\limsup_{k\rightarrow\infty}\frac{B_{k\phi,\chi_{\lambda}}dV}{N_{k}}\leq1_{\{dd^{c}\phi\geq0\}\cap\{P_{\lambda}\phi=\phi\}}MA(\phi)
\]
Finally, by Cor \ref{cor:number of eigenv} $\int\frac{B_{k\phi,\chi_{\lambda}}dV}{N_{k}}$
converges to $\int MA(P_{\lambda}\phi)$ which by formula \ref{eq:relation for ma of chi in lemma}
coincides with $1_{\{P_{\lambda}(\phi)=\phi\}}MA(\phi)$ and hence,
by basic integration theory, we conclude that the desired weak convergence
holds (in fact, one even gets the $L^{1}-$convergence of the densities
as in \cite{be1}).

\emph{Uniform convergence: }As in the the proof of Prop \ref{prop:conv of spec meas}
it is enough to consider the case when the rank of $T$ is one. We
denote by $\xi_{k}$ the linear operator on $H^{0}(X,kL)$ corresponding
to $1/k$ times differentiation wrt the real part of the generator
$V$ of $T.$ Next observe that $B_{k\phi,g}$ can be written as the
scalar product 
\[
B_{k\phi,g}(x)=\left\langle g(\xi_{k})K_{x}^{(k)},K_{x}^{(k)}\right\rangle ,
\]
 where $K_{x}^{(k)}(y)=K^{(k)}(x,y)$ is the Bergman kernel of $H^{0}(X,kL),$
i.e. the integral kernel of the orthogonal projection $\Pi_{k}$ from
the space $L^{2}(X,kL)$ of all square integrable sections of $kL$
(equipped with the $L^{2}-$norm defined by $(dV,k\phi))$. Moreover,
we may as well assume that the function $g$ is a polynomial (by a
simple approximation argument using the uniform bound $B_{k\phi}/k^{n}\leq C$).
By well-known results (see the review \cite{ze} and references therein)
$K^{(k)}(x,y)$ admits, in the case when $L$ is ample and $\phi$
has positive curvature a local asymptotic expansion of the form 
\[
K^{(k)}(x,y)/k^{n}\sim e^{k\psi(x,y)}(b_{0}(x,y)+b_{1}(x,y)k^{-1}+\ldots)
\]
in the $\mathcal{C}^{\infty}-$topology, where $\psi$ is a certain
local smooth function. Hence, fixing local coordinates centered at
$x$ and expressing the vector field $V$ in terms of the local coordinates
reveals that $g(\xi_{k})K_{x}^{(k)}(y)$ admits a local asymptotic
expansion of a similar form. Finally, setting $y=x$ we deduce in
particular that $B_{k\phi,g}(x)$ converges in the uniform topology
to some limiting function, which by the weak convergence above must
coincide with the density of $MA_{g}(\phi).$ 
\end{proof}

\subsubsection{Quantized energy functionals}

Next, following \cite{do3,bbgz}, we consider the ``quantizations''
on $\mathcal{H}_{k}^{T}$ of the functionals $\mathcal{E}$ and $J$
that we shall denote as follows: 
\begin{equation}
\mathcal{E}^{(k)}(H):=-\frac{1}{kN_{k}}\log\det(H),\,\,\, J^{(k)}(H)=-\mathcal{E}^{(k)}(H)+\mathcal{L}_{\mu_{0}}(FS(H)),\label{eq:def of quantized energy functionals}
\end{equation}
where we have identified $H$ with an Hermitian positive definite
matrix, using the reference element $H_{0}:=\mbox{Hilb}(k\phi_{0},dV)$
(writing $\mu_{0}=dV)$ and the functional $\mathcal{L}_{\mu_{0}}$
was defined in formula \ref{eq:def of J and L mu}.More generally,
$g-$analogs of these functional may be defined by setting 

\[
\mathcal{E}_{g}^{(k)}(H):=\sum_{\lambda_{i}^{(k)}\in P_{k}}g(\lambda_{i}^{(k)}/k)\mathcal{E}_{E_{\lambda_{i}^{(k)}}}^{(k)}(H),\,\,\,\mathcal{E}_{E_{\lambda_{i}^{(k)}}}^{(k)}(H):=-\frac{1}{kN_{k}}\log\det H_{|E_{\lambda_{i}^{(k)}}}
\]
Concretely, picking a base $s_{i}^{(k)}$ in $H^{0}(X,kL)$ which
is $H_{0}-$orthonormal and $H-$orthogonal and writing $H(s_{i}^{(k)},s_{i}^{(k)})=e^{-\mu_{i}^{(k)}}H_{0}(s_{i}^{(k)},s_{i}^{(k)})$
we can express 
\[
\mathcal{E}_{g}^{(k)}(H)=\frac{1}{kN_{k}}\sum_{i=1}^{N_{k}}g(\lambda_{i}^{(k)}/k)\mu_{i}^{(k)}
\]
This expression reveals that $\mathcal{E}_{g}^{(k)}(H)$ is affine
along geodesics in $\mathcal{H}_{T}$ (since the corresponding geodesics
$H_{t}^{(k)}$ are defined by the scaled eigenvalues $t\mu_{i}^{(k)}).$
Next, we introduce a $g-$analog of Donaldson's $\mathcal{L}-$functional
on the space $\mathcal{H}(L)^{T}:$ 
\begin{equation}
\mathcal{L}_{(\mu,g)}^{(k)}(\phi):=\mathcal{E}_{g}^{(k)}(\mbox{Hilb}(k\phi,\mu)=\mathcal{E}^{(k)}(\mbox{Hilb}(k\phi,\mu,g)\label{eq:def of g-donaldson l funct}
\end{equation}

\begin{prop}
\label{prop:conv of l-functional}The differential of $\phi\mapsto\mathcal{L}_{(\mu,g)}^{(k)}(\phi)$
is naturally identified with the corresponding Bergman measure: 
\[
d\mathcal{L}_{(\mu,g)}^{(k)}{}_{|\phi}=\frac{1}{N_{k}}B_{(k\phi,\mu,g)}\mu
\]
and if $\mu$ is a volume form (i.e. $\mu=dV),$ then $\mathcal{L}_{g}^{(k)}(\phi)$
converges to $\mathcal{E}_{g}(\phi),$ as $k\rightarrow\infty,$ for
any $\phi\in\mathcal{H}(X,L)^{T}.$ Moreover, in general, the functional
$\mathcal{L}_{(\mu,g)}^{(k)}(\phi)$ is concave along affine curves
in $\mathcal{H}(L)^{T}.$
\end{prop}

\subsection{$(\mu,g)-$balanced metrics }

Given a pair $(\mu,g)$ such that $\mu$ and $g\nu$ are probability
measures on $X$ and $P,$ respectively, with $\mu$ of finite energy
and $g$ continuous, we introduce the map 
\[
\mathcal{T}_{(g,\mu),k}:=\mbox{Hilb}_{(g,\mu)}\circ FS:\,\,\,\mathcal{H}_{k}\rightarrow\mathcal{H}_{k}
\]
A metric in $\mathcal{H}_{k}$ will be said to be\emph{ $(\mu,g)-$balanced
(at level $k$) }if is a fixed point of\emph{ $\mathcal{T}_{(g,\mu),k}.$
}This is thus the $g-$analog of the ordinary notion of a balanced
metric, defined by a measure $\mu$ \cite{do3}. Moreover, iterating
the map $\mathcal{T}_{(g,\mu),k}$ gives the $g-$analog of Donaldson's
iteration \cite{do3}: 
\begin{equation}
H_{m}^{(k)}:=(T_{(g,\mu),k})^{m}H_{0}\label{eq:Donaldson mu-g iteration}
\end{equation}
where $m$ is a non-negative integer (the discrete time parameter).
Just as in the ordinary case $g=1$ a metric is\emph{ }$(\mu,g)-$balanced
metric iff it is a critical point of the following functional on $\mathcal{H}_{k}^{T}$
\[
J_{(\mu,g)}^{(k)}(H)=-\mathcal{E}_{g}^{(k)}(H)+\mathcal{L}_{\mu}(FS(H)),
\]
The following result is the quantization of Theorem \ref{thm:existence uni ma eq energy}:
\begin{thm}
Let $(X,L)$ be a polarised manifold and $T$ a real torus acting
holomorphically on $(X,L)$ with moment polytope $P.$ Given a pair
$(\mu,g)$ such that $\mu$ and $g\nu$ are probability measures on
$X$ and $P,$ respectively, with $\mu$ of finite energy and $g$
continuous, there exists a \emph{$(\mu,g)-$balanced metric $H_{k}$}
for any $k$ sufficiently large, which is unique modulo scalings.
Moreover, after normalization, the corresponding Bergman type metrics
$\phi_{k}:=FS(H_{k})$ converge in $L^{1}(X)$ (or more precisely,
in energy) to the unique normalized finite energy solution $\phi$
of the Monge-Ampère equation $MA_{g}(\phi)=\mu.$\end{thm}
\begin{proof}
This is shown by adapting the proof of Theorem 7.1 in \cite{bbgz}
to our setting and reducing the problem to Theorem \ref{thm:existence uni ma eq energy}
(or rather its proof). Since a very similar argument will be carried
out in the course of the proof of Theorem \ref{thm:strong polyk implies existence and conv of quantized intro}
the details are omitted.
\end{proof}

\subsection{\label{sub:The-quantized-K=0000E4hler-Ricci}The quantized Kähler-Ricci
soliton setting }

\subsubsection{Quantized modified Futaki invariants}

Let us start by introducing the following quantized analog of the
modifed Futaki invariant defined on the Lie algebra of $\mbox{Aut\ensuremath{(X,V)_{0},}}$to
be referred to as the quantized modified Futaki invariant at level
$k:$
\[
\mbox{Fut}_{V,k}(W):=-\sum_{l=1}^{N_{k}}\exp(v_{l}^{(k)}/k)w_{l}^{(k)},
\]
where $(v_{l}^{(k)},w_{l}^{(k)})$ are the joint eigenvalues for the
commuting action of the real parts of the holomorphic vector fields
$V$ and $W$ on $H^{0}(X,-kK_{X})$ (using the canonical lifts to
$-K_{X}$ of the vector fields $V$ and $W$).
\begin{prop}
\label{prop:alg formula for fut}Given a pair $(X,V),$ consisting
of a Fano variety equipped with a holomorphic vector field $V,$ let
$W$ be a holomorphic vector field on $X$ generating a $\C^{*}-$action
and commuting with $V.$ Then 

\[
\mbox{Fut}_{V}(W)=\lim_{k\rightarrow\infty}\frac{1}{kN_{k}}\mbox{Fut}_{V,k}(W)
\]
 Moreover, the sequence in the rhs above coincides with the time derivative
of the function $t\mapsto-\mathcal{L}_{g_{V}}^{(k)}(H_{t}^{W}),$
where $H_{t}^{W}=\exp(tW)^{*}H_{0}.$ \end{prop}
\begin{proof}
Consider the commuting pair $(W,V)$ inducing an action of $S^{1}\times T.$
According to Lemma \ref{lem:alt def of fut} $\mbox{Fut}(\mathcal{X},\rho_{\mathcal{W}},\mathcal{V})=-\int_{\R^{2}}ve^{w}\nu^{(V,W)},$
where $\nu^{(V,W)}$ denotes the corresponding DH-measure on $\R^{2}.$
But then the proposition follows from Prop \ref{prop:conv of spec meas}
applied to the torus $S^{1}\times T,$ which gives that $\sum_{l=1}^{N_{k}}\delta_{(v_{l}^{(k)},w_{l}^{(k)})/k}/N_{k}$
converges to $\nu^{(V,W)}.$ Indeed integrating the latter convergence
against the function $ve^{w}$ on $\R^{2}$ concludes the proof. 
\end{proof}
We also note that in the case when $X$ is smooth there exists, for
$k$ large, a polynomial expansion 
\begin{equation}
\mbox{Fut}_{V,k}(W)=k^{n+1}\mbox{Fut}_{V}^{(0)}(W)+k^{n}\mbox{Fut}_{V}^{(1)}(W)+\cdots+\mbox{Fut}_{V}^{(n)}(W),\label{eq:asymptotic expansion of mod quant fut invariant}
\end{equation}
 where the invariant $\mbox{Fut}_{V}^{(m)}(W)$ defined by the coefficients
in the expansion above will be called the\emph{ $m$ th order modified
Futaki invariant of $W$} (by the previous proposition $\mbox{Fut}_{V}^{(0)}(W)$
is proportional to $\mbox{Fut}_{V}(W)).$ The previous expansion may
be obtained by writing 

\[
\mbox{Fut}_{V,k}(W)=-\frac{d}{dt}_{|t=0}\mbox{Tr \ensuremath{(e^{V+tW})_{|H^{0}(X,-kK_{X})}}}
\]
and evaluating the rhs using the equivariant Riemann-Roch theorem.
\begin{rem}
\label{rem:higher futaki}For $V=0$ the vanishing of $\mbox{Fut}^{(m)}(W)$
for all integers $m$ in $[0,n]$ is equivalent to the vanishing of
Futaki's higher invariants $\mathcal{F}_{Td^{(p)}}(W)$ for all integers
$p$ in $[1,n],$ which in turn is known to be equivalent to the vanishing
of Mabuchi's obstruction to asymptotic Chow semi-stability (see \cite{fu,ma0,m1}).
To see this we first note that by the equivariant Riemann-Roch theorem
$\mbox{Fut}^{(m)}(W)$ is the coefficient corresponding to $tk^{n+1-m}$
in the expansion of 
\begin{equation}
-\int_{X}e^{k(\omega_{\phi}+tf_{\phi})}\wedge\mbox{Td }(\Theta+tL_{W}),\label{eq:equivariant todd integral}
\end{equation}
where $\mbox{Td }$ is the Todd polynomial and $\Theta$ is the $\mbox{End \ensuremath{(TX)-}}$valued
Chern curvature form of the metric $\omega_{\phi}$ and $f_{\phi}$
denotes, as before, the Hamiltonian function for $W$ determined by
the canonical lift of $W$ to $L:=-K_{X}.$ Now, $u_{\phi}:=f_{\phi}-\int f_{\phi}\omega_{\phi}^{n}$
defines another Hamiltonian for $W$ satisfying the normalization
condition $\int u_{\phi}\omega_{\phi}^{n}=0$ used by Futaki \cite{fu}.
As explained in \cite{fu} the vanishing of $\mathcal{F}_{Td^{(m)}}(W)$
for all integers $m$ in $[1,n]$ is equivalent to the vanishing of
all the coefficents in the expansion obtained by replacing $f_{\phi}$
with $u_{\phi}$ in formula \ref{eq:equivariant todd integral} and
it implies that the ordinary Futaki invariant $\mbox{Fut \ensuremath{(W)}}$vanishes
(since $\mbox{Fut}(W)=c\mathcal{F}_{Td^{(1)}}(W)).$ Moreover, the
vanishing of $\mbox{Fut}^{(m)}(W)$ for $[0,n]$ also implies that
$\mbox{Fut \ensuremath{(W)}}=0$ (since $\mbox{Fut}(W)=\mbox{Fut}^{(0)}(W)).$
Finally, observing that $-\int f_{\phi}\omega_{\phi}^{n}=\mbox{Fut}\ensuremath{(W)}$
(by Lemma \ref{lem:alt def of fut}) we conclude that the vanishing
of $\mbox{Fut}^{(m)}(W)$ for all integers $m$ in $[0,n]$ is indeed
equivalent to the vanishing of Futaki's invariants $\mathcal{F}_{Td^{(p)}}(W)$
for all integers $p$ in $[1,n].$
\end{rem}

\subsubsection{Quantized Kähler-Ricci solitons and balanced metrics}

Let us next recall the definition of Donaldson's (anti-)canonical
map $\mathcal{T}_{k}$ on $\mathcal{H}_{k}$ in the ``anti-canonical''
case $L=-K_{X}.$ First, we define the anti-canonical Hilb map by
\[
\mbox{Hilb}(k\phi):=\mbox{Hilb}(k\phi,\mu_{\phi}),
\]
 where $\mu_{\phi}$ is the canonical measure on $X$ determined by
$\phi$ (see formula \ref{eq:def of mu phi}). Then the map $\mathcal{T}_{k}$
on $\mathcal{H}_{k}$ may be defined as 
\[
\mathcal{T}_{k}:=\mbox{Hilb}\circ FS
\]
 and a metric $H$ in $\mathcal{H}_{k}$ is said to be \emph{anti-canonially
balanced (at level $k)$} if it is fixed by $\mathcal{T}_{k}.$ As
conjectured by Donaldson and shown in \cite{be3} the corresponding
iteration $H_{k,m}$ on $\mathcal{H}_{k}$ converges, in the double
scaling limit where $m/k\rightarrow t,$ to the (normalized) Kähler-Ricci
flow on $\mathcal{H}(-K_{X}).$ Accordingly, we will say that a metric
$H$ in $\mathcal{H}_{k}$ is a \emph{quantized Kähler-Ricci soliton
with respect to a holomorphic vector field $V$ on $X$ }if 

\[
\mathcal{T}_{k}H=\exp(V)^{*}H
\]
where $\exp(V)^{*}$ denotes the automorphism of $\mathcal{H}_{k}$
induced by the pull-back along the time-one flow of $V$ (a similar
notion of quantized \emph{extremal metrics} was introduced in \cite{s-t}
for any ample line bundle $L,$ but defined with respect to a different
definition of the Hilb map, obtained by replacing $\mu_{\phi}$ with
$MA(\phi),$ as in \cite{d0}). Equivalently, replacing $\mbox{Hilb}$
with $\mbox{Hilb}_{V}$ defined by $\mbox{Hilb}_{V}(k\phi):=\mbox{Hilb}(k\phi,\mu_{\phi},g_{V})$
and $\mathcal{T}_{k}$ with $\mathcal{T}_{k}:=\mbox{Hilb}_{V}\circ FS$
a metric $H$ is a quantized Kähler-Ricci soliton wrt $V$ iff $H$
is fixed by the map $\mathcal{T}_{k,g_{V}},$ i.e. if it is ``anti-canonically
balanced with respect to $g_{V}$''. In this setting we define the
quantization of the modified Ding functional by 
\[
\mathcal{D}_{V}^{(k)}(H):=\mathcal{D}_{g_{V}}^{(k)}(H):=-\mathcal{E}_{g_{V}}^{(k)}(H)+\mathcal{L}(FS(H)),\,\,\,\,\mathcal{L}(\phi)=-\log\int_{X}e^{-\phi},
\]
whose critical points are quantized Kähler-Ricci solitons wrt $V.$ 
\begin{prop}
\label{prop:vartional prop of quant ding}Let $(X,V)$ be a Fano variety
equipped with a holomorphic vector field $V.$ Then the following
is equivalent:
\begin{itemize}
\item There exists a quantized Kähler-Ricci soliton at level $k$
\item The functional $\mathcal{D}_{V}^{(k)}$ is invariant under $\mbox{Aut }(X,V)_{0}$
and proper on $\mathcal{H}_{k}/\mbox{Aut }(X,V)_{0}.$
\item The functional $\mathcal{D}_{V}^{(k)}$ is coercive modulo $\mbox{Aut }(X,V)_{0},$
i.e. of at least linear growth along geodesics in $\mathcal{H}_{k}/\mbox{Aut }(X,V)_{0}.$
\end{itemize}
\end{prop}
\begin{proof}
By basic properties of convex functions on finite dimensional spaces
it will be enough to show that $\mathcal{D}_{V}^{(k)}(H)$ is strictly
convex on $\mathcal{H}_{k}/\mbox{Aut }(X,V)_{0}.$ But, as explained
above, the functional $\mathcal{E}_{g_{V}}^{(k)}(H)$ is affine along
geodesics in $\mathcal{H}_{k}$ and hence the result follows from
the well-known convexity properties of $H\mapsto\mathcal{L}(FS(H))$
(see \cite{bbgz} and reference therein).
\end{proof}

\subsubsection{Proof of Theorem \ref{thm:strong polyk implies existence and conv of quantized intro}}

We will adapt the proof of Theorem 7.1 in \cite{bbgz} to our setting.
The starting point is the following comparison inequality:
\begin{lem}
There exists a sequence $\delta_{k}\rightarrow0$ of positive numbers
such that 
\begin{equation}
J_{g}(\phi_{k})\leq(1+\delta_{k})J_{g}^{(k)}(H_{k})+\delta_{k},\,\,\,\,\,\phi_{k}:=FS(H_{k})\label{eq:comparison of energy and quant}
\end{equation}
\end{lem}
\begin{proof}
This is the $g-$analog of Lemma 7.7 in \cite{bbgz} and since the
proofs are similar we just outline the argument. First one connects
$H_{k}$ with a geodesic $H_{k}^{t}$ to the reference metric $H_{0}:=\mbox{Hilb }(k\phi_{0},dV)$
in $\mathcal{H}_{k}^{T}$ and introduces the function $f_{k}(t):=\mathcal{E}_{g}(FS(H_{k}^{t}))-\mathcal{E}_{g}^{(k)}(H_{k}).$
The latter function is convex since $\mathcal{E}_{g}^{(k)}(H_{k})$
is affine on $\mathcal{H}_{k},$ while $\mathcal{E}_{g}(HS(H_{k})$
is convex. The desired inequality is obtained by using the convexity
of $f$ and the error terms $\delta_{k}$ come from estimating the
value of $f_{k}(0)$ and the derivative $f_{k}'(0),$ using the uniform
convergence in Prop \ref{prop:conv of g-bergman} %
\footnote{For the latter smooth uniform convergence to hold we need $X$ to
be smooth. %
}. 
\end{proof}
Now, given $H_{k}$ in $\mathcal{H}_{k}^{T}$ there exists, by the
properness assumption on $\mathcal{D}_{g_{v}}$ an element $F\in\mbox{Aut }(X,V)_{0}$
such that 
\[
J_{g}(FS(F^{*}H_{k}))(1-\delta)+(\mathcal{L}-\mathcal{L}_{\mu_{0}})(FS(F^{*}H_{k}))\geq-C
\]
(where we have used that $g_{V}$ is bounded on $P$ and $\mathcal{D}_{g_{V}}$
is invariant under the action of $\mbox{Aut }(X,V)_{0},$ since the
Futaki invariants automatically vanish; compare section \ref{sec:K=0000E4hler-Ricci-solitons}).
Hence, using the inequality \ref{eq:comparison of energy and quant}
we get, for $k$ large, that 
\[
J_{g}^{(k)}(F^{*}H_{k})(1-\delta/2)+(\mathcal{L}-\mathcal{L}_{\mu_{0}})(FS(F^{*}H_{k}))\geq-2C,
\]
 i.e. 
\[
\mathcal{D}_{g}^{(k)}(F^{*}H_{k})\geq\frac{\delta}{2}J_{g}^{(k)}(F^{*}H_{k})-2C\geq\frac{\delta}{2C}J^{(k)}(F^{*}H_{k})-2C
\]
(using that $g$ is bounded in the last inequality). Now, assuming
that the quantized modified Futaki invariants $\mbox{Fut}_{V,k}(W)$
vanish for any $W\in\mbox{aut }(X,V)_{0}$ (which by the expansion
\ref{eq:asymptotic expansion of mod quant fut invariant} is equivalent
to the vanishing of the $n+1$ higher order modified Futaki invariants
of $(X,V))$ the functional $\mathcal{E}_{g_{V}}^{(k)}(H_{k})$ and
hence $\mathcal{D}_{g_{V}}^{(k)}$ is also invariant under the action
of $\mbox{Aut }(X,V)_{0}$ (by the last statement in Prop \ref{prop:alg formula for fut}).
Hence, minimizing over $\mbox{Aut }(X,V)_{0}$ in the inequalities
above gives 
\begin{equation}
\mathcal{D}_{g}^{(k)}(H_{k})\geq\delta'\inf_{F\in\mbox{Aut }(X,V)_{0}}J^{(k)}(F^{*}H_{k})-C'\label{eq:coercivity of quantizd ding}
\end{equation}
In particular, since $J^{(k)}$ is an exhaustion function for the
space $\mathcal{H}_{k}^{T}/\R$ \cite{bbgz} we conclude that $\mathcal{D}_{g}^{(k)}(H_{k})$
is proper on $\mathcal{H}_{k}^{T}/\R$ mod $\mbox{Aut }(X,V)_{0}$
and hence, by Prop \ref{prop:vartional prop of quant ding} admits
a minimizer $H_{k}$ which is unique mod $\mbox{Aut }(X,V)_{0}.$
The minimizer $H_{k}$ is uniquely determined (mod $\R)$ by the normalization
condition that the corresponding metric $\phi_{k}:=FS(H_{k})$ minimizes
$J$ on the corresponding $\mbox{Aut }(X,V)_{0}-$orbit (using the
convexity properties of $J$ as in the proof of  Theorem \ref{thm:conv of krf text}).
Moreover, by the minimizing property of $H_{k}$ we have $\mathcal{D}_{g,k}(H_{k})\leq\mathcal{D}_{g,k}(\mbox{Hilb}(k\psi)$
for any fixed smooth metric $\psi$ on $-K_{X}$ with positive curvature
and hence letting $k\rightarrow\infty$ and using the convergence
in Prop \ref{prop:conv of l-functional} gives
\begin{equation}
\mathcal{D}_{g}^{(k)}(H_{k})\leq\inf_{\psi\in\mathcal{H}(-K_{X})}\mathcal{D}_{g}=\mathcal{D}_{g}(\phi_{KRS}),\label{eq:extremal propert of balanced}
\end{equation}
 where $\phi_{KRS}$ is the unique Kähler-Ricci soliton on $X$ normalized
as above. Finally, by the inequality \ref{eq:comparison of energy and quant}
(and using that $g$ is bounded) 
\[
\mathcal{D}_{g}(\phi_{k})\leq\mathcal{D}_{g}^{(k)}(H_{k})(\phi_{k})+\delta_{k}J^{(k)}(H_{k})+\delta_{k}
\]
and since $J^{(k)}(H_{k})$ is uniformly bounded (by the inequalities
\ref{eq:coercivity of quantizd ding} and \ref{eq:extremal propert of balanced})
we conclude that $\phi_{k}$ is an asymptotically minimizing sequence
for $\mathcal{D}_{g}.$ Hence it follows, just as in the proof of
Theorem \ref{thm:proper ding mab mod aut implies existence}, that
$\phi_{k}$ converges in $L^{1}$ (and even in energy) to a minimizer
of $\mathcal{D}_{g},$ which thus concludes the proof.

\end{document}